\documentclass[12pt]{article}
\usepackage{amsmath,amssymb}
\usepackage{geometry}
\usepackage{booktabs}
\usepackage{stmaryrd}
\usepackage{pdflscape}
\usepackage{booktabs}
\usepackage{longtable} 
\usepackage{enumerate}

\allowdisplaybreaks[1]

\newtheorem{theorem}{Theorem}[section]
\newtheorem{lemma}[theorem]{Lemma}
\newtheorem{proposition}[theorem]{Proposition}
\newtheorem{corollary}[theorem]{Corollary}
\newtheorem{definition}[theorem]{Definition}

\newtheorem{problem}[theorem]{Problem}

\newcommand{\Z}{\mathbb{Z}}
\newcommand{\Q}{\mathbb{Q}}
\newcommand{\C}{\mathbb{C}}
\newcommand{\R}{\mathbb{R}}

\newcommand{\I}{\mathcal{I}}

\renewcommand{\ker}{\operatorname{Ker}}

\newcommand{\GL}{\operatorname{GL}}
\newcommand{\SL}{\operatorname{SL}}

\newcommand{\supp}{\operatorname{supp}}
\newcommand{\Ann}{\operatorname{Ann}}
\newcommand{\U}{\mathcal{U}}

\newcommand{\pmatriz}[1]{\left(\begin{array} #1 \end{array}\right)}
\newcommand{\diag}{{\rm diag}}

\newcommand{\quat}[2]{\left(\frac{#1}{#2}\right)}

\newenvironment{proof}{\par\noindent{\bf Proof.}}{$\qed$\par\bigskip}
\newcommand{\qed}{\enspace\vrule  height6pt  width4pt  depth2pt}
\usepackage{color}

\begin{document}

\title{Structure of group rings and the group of units of integral group rings: an invitation\thanks{The author is supported in part by Onderzoeksraad of Vrije
Universiteit Brussel, Fonds voor Wetenschappelijk Onderzoek
(Belgium, grant G016117) and the International Centre for Theoretical Sciences (ICTS) during a visit for participating in the program-  Group Algebras, Representations and Computation (Code:  ICTS/Prog-garc2019/10).
\newline {\em 2020 MSC}:  16S34, 16U40, 16U60, 16H10, 20C05.
\newline {\em Keywords:} group ring, unit, idempotent, rational representations.
}}
\author{E. Jespers\\ $\;$\\
{\em Dedicated to  I.B.S. Passi on  the occasion of his 80th birthday}}
\date{}

\maketitle

\begin{abstract}
During the past three decades fundamental progress has been made on  constructing large torsion-free subgroups  (i.e. subgroups of finite index) of the unit group $\U (\Z G)$ of the integral group ring $\Z G$ of a finite group $G$. 
These constructions rely on explicit constructions of units in $\Z G$ and proofs of main results make use  of the description of the Wedderburn components of the rational group algebra $\Q G$. The latter relies on explicit constructions of primitive central idempotents and the rational representations of $G$.  It turns out that the existence of reduced  two degree  representations play a crucial role. Although the unit group is far from being understood, some structure results on this group have been obtained.
In this paper we give a survey of some of the fundamental results and  the essential needed techniques.
\end{abstract}

\section{Introduction}

The integral group ring $\Z G$ of a finite group $G$ is a ring that,  in some sense, solely is based on 
the defining group $G$. So, the group ring $\Z G$ is a tool that serves as a meeting place between group 
and ring theory. The defining group $G$ is a subgroup of the unit group $\U (\Z G)$, and hence, if the interplay 
works well, there should be strong relation between the group $G$ and the unit group $\U (\Z G)$. The integral 
group ring  $\Z G$ is an order in the finite dimensional semisimple rational group algebra $\Q G$. The latter can 
be well described as a product of matrix algebras over division algebras by using strong structure theorems and 
the rational representations. However, since there are many many orders in $\Q G$ it is a difficult problem to 
rediscover $\Z G$ from such a matrix decomposition.

Several books have been dedicated to algebraic structural topics in noncommutative  group rings: A. Bovdi \cite{BovdiBook}, N. Gupta \cite{GuptaBook},
G. Karpilovsky \cite{KarpilovskyBook1,KarpilovskyBook2,KarpilovskyBook3,KarpilovskyBook4,KarpilovskyBook5,KarpilovskyBook6},
G. Lee \cite{LeeBook}, D.S. Passman \cite{PassmanBook0,PassmanBook1,PassmanBook2}, I.B.S. Passi \cite{PassiBook}, W. Plesken \cite{PleskenBook}, C. Polcino Milies \cite{PolcinoBook},  I. Reiner \cite{ReinerBook}, K.W. Roggenkamp \cite{RoggenkampBook}, K.W Roggenkamp and M.J. Taylor \cite{RoggenkampTaylorBook}, S.K. Sehgal \cite{SehgalBook3}, M. Taylor \cite{TaylorBook}, A.E. Zalesskii, andA.V.  Mihalev \cite{ZalesskiiMihalevBook}.
The following books are specifically dedicated to the study of units of integral  groups rings of finite groups: C. Polcino Milies and S.K. Sehgal \cite{PolcinoSehgalBook}, S.K. Sehgal \cite{SehgalBook2} and E. Jespers and \'A del R\'{\i}o \cite{bookvol1,bookvol2}. In \cite{bookvol1,bookvol2} the state of the art is given (up to 2016) on the construction of large (torsion-free) subgroups of $\U (\Z G)$, for finite groups $G$, and on structural results of $\U (\Z G)$.
The aim of this paper is to give the reader an intuition and thus to invite researchers to the topic. This is done by guiding the reader through the following topics: (1) the essential constructions
of units and large subgroups, (2)  surveying the fundamental methods needed (such as the construction of primitive central idempotents of the rational group algebra $\Q G$ and the description of its simple components),  (3) constructions of large central subgroups, (4) constructions of free groups, (5) structure results, in particular  abelianisation and amalgamation results. The last topics are recent results that are not included in \cite{bookvol1,bookvol2}.
The other topics covered  are based on \cite{bookvol1,bookvol2} and the reader should consult these books in case no explicit references are given. For further references we refer to the bibliography in these books. For some topics we will include references to some recent results, without the aim of being comprehensive.  Only few proofs will be included.

\section{The unit group versus the isomorphism class}
 
 Let $R$ be a ring and $G$ a group. The group ring $RG$ is  the free $R$-module with basis $G$, i.e. it consists of all formal
 sums
  $\sum_{g\in G} r_g g$,
 with only a finite number of coefficients $r_g\in R$ different from $0$,
 and with addition defined as
   $$\sum_{g\in G} r_g g + \sum_{g\in G} r'_g g =\sum_{g\in G} (r_g + r'_g) g,$$
and a product that extends the products of both $R$ and $G$, i.e.
   $$\left( \sum_{g\in G} r_g g\right) \; \left( \sum_{h\in G} s_h h \right) = \sum_{x\in G} \left( \sum_{g,h\in G, \; gh=x} r_g s_h\right) x.$$
 The support of an element $\alpha =\sum_{g\in G} r_g g\in RG$ is the finite set
$\supp (\alpha ) =\{ g\in G \mid r_g \neq 0\}$.

The augmentation map of $RG$ is the ring homomorphism
 $$\text{aug} : RG \rightarrow R: \sum_{g\in G} r_g g \mapsto \sum_{g\in G} r_g.$$
 More generally, for a normal subgroup $N$ of $G$, the augmentation map modulo $N$ (also called the relative augmentation map) is the ring homomorphism 
  $$\text{aug}_{G,N,R} RG \rightarrow R(G/N): \sum_{g\in G} r_g g \mapsto \sum_{g\in G} r_g (gN).$$
 The kernel of $\text{aug}_{G,N,R}$ is called  the augmentation ideal of $RG$ modulo $N$. If the ring $R$ is clear then  from the context we simply denote this map as
 $\text{aug}_{G,N}$. It  readily is verified that
  $$\ker (\text{aug}_{G,N,R}) = \sum_{n\in N} (n-1) RG =\sum_{n\in N} RG (n-1).$$
If, furthermore, $N$ is finite  then 
   $$\widetilde{N} =\sum_{n\in N}n$$ 
 is a central element of $RG$ and $\widetilde{N}(1-n)=0$ for all $n\in N$.
 Hence,
   $$\widetilde{N}^{2}=|N|\widetilde{N} .$$
 Moreover,
 $$\ker (\text{aug}_{G,N,R})= \Ann_{RG}(\widetilde{N}) =\{ \alpha \in RG \mid \alpha \widetilde{N}=0\}.$$
If, furthermore $|N|$ is invertible in $R$ then 
 $$\widehat{N}=\frac{1}{|N|} \widetilde{N}$$ is a central idempotent in $RG$ and
 $$RG = RG \widehat{N} \oplus RG (1-\widehat{N}) \quad \mbox{ and } \quad R(G/N) = RG \widehat{N}.$$
 
 The unit group of a ring $R$, denoted $\U (R)$, is the group
  $$\U (R) = \{ u\in R \mid uv =1 =vu, \mbox{ some } v\in R\} .$$
 Our interest mainly goes to the unit group of the integral group ring $\Z G$  of a finite group $G$.
 Of course, if $\alpha =\sum_{g\in G} r_g g \in \U (\Z G)$ then $\text{aug} (\alpha )=\sum_{g\in G} r_g =\pm 1$.
 A unit $\alpha \in \Z G$ is said to be {\it normalized} if $\text{aug} (\alpha ) =1$. The group consisting of all normalized units of $\Z G$ is  denoted  $\U_1 (\Z G)$.
 Clearly
  $$\U (\Z G) =\pm \U_{1}(\Z G) .$$
 Note that if $R$ is a commutative ring then the group ring $RG$ is endowed with an involution $*$ (often called the classical involution)
   $$* : RG \rightarrow RG : \sum_{g\in G} r_g g \mapsto \sum_{g\in G} r_g g^{-1}.$$
 
 The integral group ring is the ring that links group theory to ring theory. 
 One hence has a natural fundamental question: the isomorphism problem for integral group rings of finite groups $G$ and $H$:
 \begin{center}
  Is it true that  the a ring ismorphism $\Z G \cong \Z H$\\ implies a group isomorphism $G\cong H$?  \;\;\;\;\;\;(ISO)
  \end{center}
 This question first was posed by Higman in his thesis \cite{Higman}.
  The following proposition is a remarkable fact for group rings: an integral group ring isomorphism is equivalent with unit group isomorphism. To prove this, we 
 first  need a lemma. Of course $\Z G$ is a subring of the rational group algebra $\Q G$; and thus one can talk of ($\Q$-) independent elements in $\Z G$.
 
 \begin{lemma}\label{IndependentFiniteSubgroup}
 Let $G$ be a finite group.  The following properties hold for a subgroup $H$ of $G$.
 \begin{enumerate}
 \item (Berman) If $\alpha= \sum_{g\in G} z_g g$ is a unit of finite order in $\U_{1}(Z G)$ such that $z_1\neq 0$ (with $1$ the identity of $G$) then $\alpha = 1$. In particular, if $\alpha$ is a normalized central unit of finite order then $u\in Z(G)$.
 \item $H$ is a set of independent elements;
 in particular, $|H|\leq |G|$.
 \item If $|H|=|G|$ then $\Z G =\Z H$.
 \end{enumerate}
 \end{lemma}
 \begin{proof}
 We only prove the second and third part.
 
 (2) Assume that $ \sum_{h\in H} z_h h=0$, with each $z_h\in \Z$. Let  $x\in H$. Since, by assumption, $H$ is finite, also $hx^{-1}$ is unit of finite order in $\U_{1}(\Z G)$.
 Clearly, $hx^{-1}\neq 1$ if $h\neq x$. Hence, by part (1), the coefficient of $1$ in $hx^{-1}=0$. 
 Since the coefficient of $1$ in $ \sum_{h\in H} z_h hx^{-1}$ equals $z_x$, we conclude 
 that $z_x=0$. Since $x$ is arbitrary, part (2) follows.

 (3) Assume $H$ is a finite subgroup of $\U_{1}(\Z G)$ and $|H|=|G|$. By part (2) the elements of $H$ are independent and thus $\Q G =\Q H$. So,
 $\Z H \subseteq \Z G$ and $n\Z G \subseteq \Z H$ for some positive integer $n$. It remains to show that if $g\in G$ then  $ g\in  \Z H$. So, let $g\in G$ and write
 $ng = \sum_{h\in H} z_h h$, with each $z_h \in \Z$. 
 Then $ngh^{-1} =z_h + \sum_{h'\in H, h'\neq h} z_h (h'h^{-1})$.
 As $1\neq h'h^{-1}$ is periodic, it follows from part (1) that  the coefficient  of $1$ of $h'h^{-1}$ is $0$. Therefore, the coefficient of $1$ in $ngh^{-1}$ equals $z_h$. Hence it has to be divisible by $n$. As $h$ is arbitrary we have shown that for every  $h\in H$ in the support of $ng$ we have that $n|h$. Consequently, $ng\in n\Z H$ and thus $g\in \Z H$, as desired.
 \end{proof}

 \begin{proposition}
 Let $G$ and $H$ be finite groups. The following statements are equivalent.
 \begin{enumerate}
 \item $\Z G  \cong \Z H$ (ring isomorphism),
 \item $\U_1 (\Z G) \cong \U_1 (\Z H)$ (group isomorphism),
 \item $\U (\Z G) \cong \U (\Z H)$ (group isomorphism).
 \end{enumerate}
 \end{proposition}
 
 \begin{proof}
 Clearly (1) implies (3).
 For the other implication it is useful to turn any isomorphism into a normalized isomorphism. This is done as follows, for any commutative ring $R$.
 Let $f: \U(R G) \rightarrow \U(R H)$ be a group isomorphism. Define
  $$f^{*}: \U(RG) \rightarrow \U(RH) : \sum_{g\in G} r_g g \mapsto \sum_{g\in G} r_g \left( \text{aug}_{H} (f(g)) \right)^{-1} f(g).$$
 It is readily verified that $f^{*}$ is an  isomorphism that preserves augmentation, i.e. $\text{aug}_{H} (f^{*} (g)) =1$ for all $g\in G$ and thus 
 $\text{aug}_{H}(f^{*}(\alpha )) =\text{aug}_{G}(\alpha)$ for all $\alpha \in \U(RG)$.
 Hence, (3) implies (2).
 
 Now, assume (2) holds, i.e. assume  $f: \U_1 (\Z H) \rightarrow  \U_1 (\Z G)$ is a group isomorphism. Then $f(H)$ is  a finite subgroup of $ \U_{1}(\Z G)$ that is isomorphic to $H$. Hence, by Lemma~\ref{IndependentFiniteSubgroup}, $|H|=|f(H)|\leq |G|$. Similarly, $|G|\leq |H|$ and thus $|H|=|G|$.
 Furthermore, by Lemma~\ref{IndependentFiniteSubgroup}, the  $\Z f(H)= \Z G$  and thus we obtain an isomorphism $\Z H \rightarrow \Z G$, as desired.
 \end{proof}
  
Hence, (ISO) is equivalent with
  \begin{eqnarray*}
 \U (\Z G) \cong \U(\Z H) & \Longrightarrow & G\cong H. \quad \mbox{(ISO')}
 \end{eqnarray*}

It thus is a fundamental problem to describe the unit group of the integral group ring of a finite group. 
It is a hard problem to fully describe this group and hence  one often focusses on describing a large subgroup, i.e. 
a subgroup of finite index. Preferably one would like a torsion-free subgroup of index exactly $|G|$. In other words one has the following problem.
 
 \begin{problem}\label{Problem1Iso}: Let $G$ be a finite group. Does there exists a torsion-free normal subgroup of finite index, say $N$ such that
 $|\U_1 (\Z G)/N|= |G|$. This means that 
    $$\U_{1}(\Z G) = N \rtimes G,$$
 a semidirect product of groups (i.e. the inclusion $G\rightarrow \U_1 (\Z G)$ splits).
 \end{problem}

 It easily is verified (see \cite[Chapter 4]{SehgalBook2}) that an  affirmative answer to Problem~\ref{Problem1Iso} gives an affirmative answer to (ISO).   In case $G$ is a nilpotent group it is sufficient to check that there is 
 a normal complement. 
  
 Note that Roggenkamp and Scott gave a metabelian counter example to the problem (nevertheless,  Withcomb proved (ISO) holds for finite metabelian groups). However, because of the link with (ISO), it remains a challenge to determine classes of groups for which there is a positive answer. 
A positive answer to  Problem~\ref{Problem1Iso}  also has been proven for finite groups $G$ having an abelian normal subgroup $A$ such that either $G/A$ has exponent dividing 4 or 6, or $G/A$ is abelian of odd order (by Cliff-Sehgal-Weiss). We refer the reader to \cite{SehgalBook2} for proofs.
 
 However, the general problem remains open. Sehgal in  \cite[Problems 29 and 30]{SehgalBook2} stated two challenging problems. 
 
 \begin{problem}\label{Problem2930} (Sehgal)
 Does Problem~\ref{Problem1Iso} have an affirmative answer if $G$ is a finite nilpotent group? Even in case $G$ has nilpotency class three the answer is not known. For class two the answer is affirmative.
 \end{problem}
 
Nevertheless, using other methods,  (ISO) has been proven for the following classes of finite groups: metabelian groups (Whitcomb), nilpotent groups (Roggenkamp and Scott) and simple groups (Kimmerle, Lyons, Sandling).
Hertweck \cite{HertweckIso} has given a counter example to the isomorphism problem. It is a group 
of  order $2^{21} 97^{28}$, with a normal Sylow $97$-subgroup and the group has derived
length $4$.

\section{Construction of units}

In order to study the unit group  $\U (\Z G)$, with $G$ a finite group, one first would like to know  some generic construction of units. Apart from the trivial units there are two main constructions: the Bass units (introduced by Bass) and the bicyclic units (introduced by Ritter and Sehgal).

\vspace{12pt}
{\it Trivial units}\\ Clearly $\pm G \subseteq \U (\Z G)$. The elements of $\pm G$  are called the trivial units.

\vspace{12pt}
{\it Unipotent units and bicyclic units}\\
 Let $R$ be an associative ring with identity element $1$. Let $\eta$ be a nilpotent element of $R$, i.e. $\eta^{k}=0$ for some positive integer $k$.
Then 
 $$(1-\eta )(1+\eta + \cdots + \eta^{k-1})=1 = (1+\eta + \cdots + \eta^{k-1})(1-\eta).$$
 So, from nilpotent elements one can construct units. Note that the rational group algebra $\Q G$ has no non-zero nilpotent elements if and only if
 $\Q G$ is a direct sum of division algebras. Hence, for most finite groups the group algebra has nilpotent elements (the only exceptions being the abelian groups and the Hamiltonian groups of order $2^m t$, with $t$ an odd number such that the multiplicative order of $2$ modulo $t$ is odd).
 
One can construct nilpotent elements from almost idempotent elements $e\in R$ (i.e. $e^2=ne$ for some positive integer $n$).
For any $r\in R$,
   $$\left( (n-e) re\right)^{2}=0$$
   and thus
    $$1+(n-e)re$$
is a unipotent unit (with inverse $1-(n-e)re$).

Let $G$ be a finite group and let $e$ be an idempotent in $\Q G$
(recall that $\Z G$ only contains $0$ and $1$ as idempotents, see Section 6). Let $n_e$ be the smallest positive integer such that $n_e e\in \Z G$.
Then, for $h\in G$,
  $$ b(h,e) = 1 + n_e^{2} (1-e)he \quad \mbox{ and } \quad b(e,h) =1+n_e^{2} eh(1-e) $$
are unipotent units in $\Z G$. They are called {\it generalized bicyclic units}.

In rational group algebras one can easily construct idempotents. Indeed, let $g\in G$ be an element of order $n$. Then,
  $$\widehat{g} =\widehat{\langle g \rangle} =\frac{1}{n} \widetilde{\langle g \rangle} = \frac{1}{n} \widetilde{g}= \frac{1}{n} \sum_{0\leq i \leq n-1} g^{i}$$
is an idempotent in $\Q G$ and $\widetilde{g}$ is an almost idempotent in $\Z G$.
The units
   $$ b(h,\widetilde{g}) = 1 +  (1-g)h\widetilde{g}  \quad \mbox{ and } \quad b(\widetilde{g},h) =1+\widetilde{g} h(1-g) $$
are called the {\it bicyclic units} of $\Z G$.
Obviously, $ b(h,\widetilde{g})^{-1} =b(-h, \widetilde{g})$.
Note that a bicyclic unit  $b(h,\widetilde{g})$ is trivial unit if and only if $h\in N_G(\langle g \rangle)$; otherwise it is a unit of infinite order.
 Note that  $b(h,\widehat{g})=b(\alpha, \widetilde{g})$ for some $\alpha \in \Z \langle g \rangle$.

 \vspace{12pt}
 {\it  Cyclotomic units and Bass units}\\
 Let $R$ be an associative ring and $x$ a unit of finite order $n$. Let $k$ and $n$ be relatively prime positive integers and let $m$ be a positive integer such that $k^m \equiv 1$ mod $n$.
 Then
  $$u_{k,m}(x) =(1+x + \cdots + x^{k-1})^{m} + \frac{1-k^{m}}{n} (1+ x + \cdots + x^{n-1})$$
 is an invertible element in $R$ with inverse $u_{l,m}(x^k)$, where $l$ is a positive integer such that $kl \equiv 1$ mod $n$.
 Note that if $R$ is a domain and $x\neq 1$ then $(1-x)((1+x+\cdots + x^{n-1})=0$ implies that $1+x+\cdots + x^{n-1} =0$ and thus, in this case,
 $u_{k,m}=(1+x + \cdots + x^{k-1})^{m}$. If, furthermore, $R$ is a field then
 $u_{k,m}=(1+x + \cdots + x^{k-1})^{m}=\left(\frac{1-x^{k}}{1-x}\right)^{m}$.
 The unit $\frac{1-x^{k}}{1-x}$ is called a {\it cyclotomic unit} and is denoted 
    $$\eta_{k}(x) =\frac{1-x^{k}}{1-x}.$$ 
 Note that
 $(\eta_{k}(x))^{-1} =\eta_{l}(x^{k})$, where $l$ is a positive integer such that $lk\equiv 1$ mod $n$.
 Hence $\eta_{k}(x)\in \U (\Z [x])$.

 We also remark that if $x\in R$ is a unit of finite odd order then $-x$ has even order and $u_{k,m}(-x) = (1-x+x^{2} + \cdots + (-1)^{k-1})^{m}$.
 Such units are called {\it alternating units} in integral group rings \cite{SehgalBook2}.
 
 Let $G$ be a finite group. We remark that, for $g\in G$,
  $$u_{k,m}(g) = u_{k',m}(g) \quad \quad \text {if } k\equiv k' \text{ mod } |g|.$$
 Hence, in the definition of $u_{k,m}(g)$  we may assume that $1< k<|g|$.
 The units $u_{k,m}(g)$, with $g\in G$ and $(k,|g|)=1$ are called the {\it Bass units} of $\Z G$. These were introduced by Bass in \cite{Bass}.
 One can also show that
   $$u_{k,m}(g) u_{k,m_{1}}(g) = u_{k,m+m_{1}}(g).$$

 We now show that almost all Bass units are of infinite order.

 \begin{lemma}\label{BassTorsionL}
 Let $G$ be a finite group and $g\in G$.
A Bass unit $u_{k,m}(g)$ is torsion if and only if $k\equiv \pm 1 \mod |g|$.
\end{lemma}

\begin{proof}
Let $n=|g|$ and let $u=u_{k,m}(g)$. If $k\equiv 1$ mod $n$ then $u=1$ and the result is
clear in this case. 

So, assume that $k\not\equiv 1$ mod $n$ and in particular $n>1$. 
If $k=n-1$ and $m=2$ then $u=(\widetilde{g}-g^{n-1})^2-\frac{1-(n-1)^2}{n}\widetilde{g}
= g^{-2}$. If $k\equiv -1$ mod $n$ then
$m$ is a multiple of $2$ and thus 
$u=u_{n-1,m}(g)=u_{n-1,2}(g)^{\frac{m}{2}}
= g^{-m}$. 
This proves that if $k\equiv
\pm 1$ mod $n$ then $u$ is torsion.

Conversely, assume that $u$ is torsion. 
Let $\zeta$ be a complex root of unity of order $n$.
By the Universal Property of Group Rings, the
group isomorphism $\langle g \rangle \rightarrow
\langle \zeta \rangle$, mapping $g$ to $\zeta$, extends to
a ring homomorphism $f:\Z G\rightarrow \C$. As
$n>1$, $f(\widetilde{g})=0$ and therefore
$f(u)=\eta_k(\zeta)^m$. Since $u$ is torsion,
$f(u)$ is a root of unity, hence so is
$\eta_k(\zeta)$. This implies that
$|\zeta^k-1|=|\zeta -1|$. Thus  $\zeta$ and
$\zeta^k$ are two vertices of a regular polygon
with $n$ vertices so that $\zeta$ and $\zeta^k$
are at the same distance to $1$. This implies
that $\zeta^k$ is either $\zeta$ or
$\overline{\zeta}=\zeta^{-1}$. Then $k\equiv \pm 1$ mod $n$, as desired.
\end{proof}

We have  introduced two type of units: the Bass units and the bicyclic units. 
The constructions of these are based on the cyclotomic units and unipotent units.
These units are of great importance for the unit group $\U (\Z G)$. The main reason being the following results.

\begin{theorem}
Let $\xi$ be a complex root of unity. The cyclotomic units of $\Z [\xi ]$ generate a subgroup of finite index in $\U (\Z [\xi ])$.
\end{theorem}

For a ring $R$ and positive integer $n$, we denote by 
   $$e_{ij}(r)\in M_n(R)$$ 
the unipotent matrix $1+rE_{ij}$, where $E_{ij}$ is the elementary matrix that has only one nonzero entry (at position $(i,j)$) and this entry equals $1$. A useful formula is 
   $$E_{ij}E_{kl}=\delta_{jk}E_{il}.$$

\begin{proposition}$\;$\newline
\vspace{-11pt}
\begin{enumerate}
\item
The group $SL_{n}(\Z)$ is generated by the matrices $e_{ij}=1+E_{ij}$ with $i\neq j$.
\item (Sanov) Let $z_1,z_2\in \C$ such that $|z_1z_2|\geq 4$ then $\langle e_{12}(z_1), e_{21}(z_2)\rangle$ is a free group of rank 2.
\item The group $\left\{ \left( \begin{array}{cc} a&b\\ c&d\end{array} \right) \in SL_{2}(\Z ) \mid a\equiv d \equiv  1\mbox{ mod } 4,\;\right\}$
is a free group of rank $2$ generated by 
$e_{12}(2) $ and $e_{21}(2)$.
\end{enumerate}
\end{proposition}

Let $R$ be a commutative Noetherian domain with field of fractions $F$ and let $A$ be a finite dimensional $F$-algebra.
A full $R$-lattice in a finite dimensional  $F$-vector space $V$ is a finitely generated $R$-submodule of $V$ (i.e. an $R$-lattice in $V$) that contains a 
basis of $V$.
An $R$-order in $A$ is a subring of $A$ which also is a full $R$-lattice in $A$.
A $\Z$-order will be simply called an order. Because $\Z$ is a PID, an order contains a $\Z$-basis  and this obviously also is a $\Q$-basis of $A$.
Clearly,  $M_{n}(R)$ is an $R$-order in $M_{n}(F)$. Also, if $\mathcal{O}$ is an order in $A$
then $M_{n}(\mathcal{O})$ is an order in $M_{n}(A)$. The integral quaternions $\left( \frac{-1,-1}{\Z }\right) $ is a order in the division algbera $\left( \frac{-1,-1}{\Q}\right)$ (see Section 4).
Obviously, if $G$ is a finite group then $\Z G$ is an order in $\Q G$

 With ``elementary methods'' (see \cite[Chapter 1]{bookvol1}) one can calculate the unit group of some some well known rings. By $\xi_n$ we denote a complex root of unity of order $n$.
 \begin{enumerate}
 \item $\U (\Z )=\{ -1, 1\}$.
 \item
 $\U (\Z [i]) =\{ 1,-1, i, -i\}$.
 \item
 $\U (\Z [\xi_{3}])=\{ \pm 1, \pm \xi_3 , \pm \xi_3^2 \}$.
 \item 
 $\U (\Z [\xi_{6}])=\langle \xi_6 \rangle$.
 \item
 $\U (\Z [\xi_{8}])=\langle \xi_8\rangle \times \langle 1+\sqrt{2} \rangle = \langle \xi_8 \rangle \times \langle \eta_3(\xi_8) \rangle 
 =C_8 \times C_{\infty} $ and $\eta_{3}(\xi_8) = 1+\xi_8 + \xi_8^2$.
 \item
 $U(\Z C_5) = \pm C_5 \times \langle g+g^4 -1\rangle$, where $C_5 =\langle g \mid g^5=1 \rangle$.
 \item $U_{1} (\Z C_8) = C_8 \times \langle u_{3,2}(g)\rangle = C_8 \times C_{\infty}$, where $C_8=\langle g \mid g^8=1\rangle$.
 \item
 $\U \left(\left( \frac{-1,-1}{\Z }\right) \right) =Q_8$, where $Q_8$ is the quaternion group of order $8$.
 \item (Higman) $\U_{1} (\Z Q_8) =Q_8$
 \item $\U_1 (\Z D_8) =B \rtimes D_8$, where $B$ is the subgroup generated by the bicyclic units. Furthermore,
 $B$ is a free group of rank $3$
 \end{enumerate}
 
 The previous list contains several examples of unit groups that are finite. Actually all relevant groups are included in these examples as shown by the
 following result of Higman.

\begin{theorem} (Higman)
The following conditions are equivalent for a finite group $G$.
\begin{enumerate}
\item $\U_1 (\Z G)$ is finite.
\item $\U_1 (\Z G) =G$.
\item $G$ is an abelian group of exponent dividing $4$ or $6$, or $G\cong Q_8 \times E$, with $Q_8$ the quaternion group of order $8$ and $E$ an elementarry abelian $2$-group (i.e. a direct product of copies of the cyclic group $C_2$ of order $2$).
\end{enumerate}
\end{theorem}

For the proof of this result  one can make use of the Bass units and bicylic units and of the fact that if $\U_1 (\Z G)$ is finite then so is the unit group $\U_1 (\Z (G\times C_2))$.

So, for almost all finite groups $G$, the unit group  $\U_1 (\Z G)$ is infinite.
Actually one can prove that if the unit group $\U (\Z G)$ is infinite and $G$ is not abelian and not a  Hamiltonian group, i.e. not all subgroups are normal,
then $\U (\Z G)$ contains a free group of rank 2 generated by two bicyclic units.

To prove this result  Salwa \cite{Salwa} showed a more general result.

\begin{theorem}
Let $R$ be a torsion-free ring and $a,b\in R$ such that $a^2=b^2=0$. Then the group $\langle 1+a , 1+b\rangle$ is free if and only if either $ab$ is transcendental or $ab$ is algebraic (over $\Q$) and one of the eigenvalues $\lambda$ of $ab$ is a free point (that is $\langle e_{12}(1),e_{21}(\lambda)\rangle$
is a free group).

\end{theorem}
\begin{proof}
Since we are mainly interested in group rings of finite groups in this paper, we will indicate a proof in the case $ab$ is algebraic. Without loss of generality, we may assume that $R=\Z [a,b]$, that is, $R$ is a $\Z$-module and as a ring it is generated by $\Z$ and $a$ and $b$.  Let $A =\Q [a,b]$. Let $J=J(A)$ denote the Jacobson radical of $A$. By assumption $ab$ is algebraic over $\Q$ and thus $\Q [a,b] =\Q [ab] + b\Q [ab] + \Q [ab]a + b\Q [ab] a$ is finite dimensional over $\Q$ and $J$ is a nilpotent ideal. As $(1+J^n) / (1+J^{n+1})$ is central in $(1+J)/(1+J^{n+1})$ we deduce that $1+J$ is a nilpotent group and hence so is
$(1+J)\cap \langle 1+a , 1+b\rangle$. Thus $\langle 1+a , 1+b\rangle$ is free if and only if so is $\langle 1+\overline{a} ,1+\overline{b}\rangle \subseteq \U (A/J)$.

Now, let $\rho$ denote the regular representation of $A$ over $\Q$. Let $\lambda_1 , \ldots , \lambda_k$ be the non-zero eigenvalues of $\rho (ab)$.
One then can prove that
 $$\overline{A} =A/J \cong \Q^m \oplus \prod_{i=1}^{n} M_{2}(\Q (\mu_i )),$$
 with $\{ \mu_{1} , \ldots , \mu_{n}\} =\{ \lambda_1 , \ldots , \lambda_k\}$ and the isomorphism associates $1+\overline{a}$ and $1+\overline{b}$ with
 $(1, \ldots , 1, e_{12}(1),\ldots , e_{12}(1))$ and $(1,\ldots, 1, e_{21}(\mu_1 ), \ldots , e_{21}(\mu_k ))$ respectively.
 It follows that  $\langle 1+a , 1+b\rangle$ is free if and only if each $\langle e_{12}(1), e_{21}(\mu_i )\rangle$ is a free group and thus the result follows.
\end{proof}

If $G$ is a finite  group of order $n$  and $R$ is a commutative ring then the trace function of $RG$ is the map $T:RG \rightarrow R$ associating to each element of $RG$ the coefficient of $1$, i.e. $T(\sum_{g\in G} r_g g) = r_1$. Let $\rho$ denote the regular representation given by left multiplication. Then
$T(x) =\frac{1}{n} \text{tr}(\rho (x))$, for every $x\in RG$. So, in case $R=\C$ then $T$ can be considered as the restriction of $\frac{1}{n} \text{tr}$ to $\C G$.
Salwa also proved the following

Recall that a trace function $T$ on a complex algebra $A$ is a $\C$ linear map $A\rightarrow \C$ such that $T(ab)=T(ba)$ for $a,b\in A$, $T(e)$ is a positive real number for all non-zero idempotents $e\in A$ and $T(a)=0$ for every nilpotent element $a\in A$.

\begin{proposition}
Let $A$ be a complex algebra and let $T$ be a trace function on $A$. If $a,b\in A$ are such that $a^2=b^2=0$ and $|T(ab)|\geq 2T(1)$ then
$\langle 1+a, 1+b\rangle$ is a free group.
\end{proposition}

\begin{theorem} (Marciniak-Sehgal)
Let $G$ be a finite group and let $u$ be a non-trivial bicyclic unit then $\langle u,u^{*}\rangle$ is a free group of rank 2.
\end{theorem}

\begin{proof}
Let $T$ be the above mentioned trace map om $\C G$ and let $u=b(g,\widetilde{h})\neq 1$ with $g,h\in G$. Let $a=u-1=(1-h)g\widetilde{h}$ and $b=a^{*}=\widetilde{h}g^{-1}(1-h^{-1})$. Then $ba =\widetilde{h} g^{-1} (2-h-h^{-1})g\widetilde{h} = \widetilde{h} (2-z-z^{-1})\widetilde{h}$, with $z\not\in \langle h \rangle$. Therefore, $T(ab)=T(ba) =2|h|\geq 4 =4T(1)$. Hence, $\langle u =1+1, u^{*}=1+b\rangle$ is free by the previous Proposition. This proves the result for $u=b(g,\widetilde{h})\neq 1$. A similar argument deals with $u=b(\widetilde{h},g)$.
\end{proof}

\section{ Primitive central idempotents and simple components}

In this section we discuss the decomposition of the semisimple rational group algebra $\Q G$ of a finite group into a product of
simple components, the so called Wedderburn components. 
We begin by recalling the fundamental theorem describing semisimple rings.

\begin{theorem} (Wedderburn-Artin)
A ring $R$ is semisimple if and only if $R$ is isomorphic to a finite direct product of matrix rings over division rings.
\end{theorem}

So, if $R$ is a semisimple algebra then
 $$R=Re_1 \times \cdots \times Re_k \cong M_{n_1}(D_1) \times \cdots \times M_{n_{k}}(D_{k}),$$
 where $n_1, \ldots , n_k$ are positive integers, each $D_i$ is a division ring and each $e_i$ is a primitive central idempotent.

\begin{theorem} (Maschke's Theorem)
Let $R$ be a ring and $G$ a group. The group ring $RG$ is semisimple if and only if $R$ is semisimple, $G$ is finite and $|G|$ is invertible in $R$ (i.e. $|G|r=1$ for some $r\in R$). In case $R$ is a field, the latter means that $|G|$ is not a multiple of the characteristic of $R$
\end{theorem}

If $F$ is a field and $G$ is a finite group such that $FG$ is semisimple then
 $$FG=FGe_1 \times \cdots \times FGe_k\cong M_{n_1}(D_1) \times \cdots \times M_{n_{k}}(D_{k}),$$
 and each simple algebra $FGe_i$ is as an $F$-algebra generated by the finite group $Ge_i =\{ ge_i \mid g\in G\}$.
 Clearly, $Ge_i \cong G/S_{G} (e_i)$, where $S_{G}(e_i)=\{ g\in G \mid ge_i =e_i\}$, the stabiliser of $e_i$ in $G$.
 In case $G$ is abelian then, of course, each $n_i=1$ and $D_i$ is a field. Since finite subgroups of a field are cyclic, we get that, in this case
 each $FG e_i = F(\xi_{n_i})$, where $\xi_{n_i}$ is a primitive $n_i$-th root of unity in the algebraic closure of $F$.
 One can then prove the following result.
 
 \begin{theorem} (Perlis-Walker)
 Let $G$ be a finite abelian group and $F$ a field of characteristic $0$. Let $k_d$ denote the number of cyclic subgroups of $G$ of order $d$. Then
   $$FG \cong \prod_{d,\; d| |G|} F(\xi_d)^{k_d \frac{[\Q (\xi_d):\Q]}{[F(\xi_d):F]}}.$$ 
In particular,
  $$\Q G\cong  \prod_{d,\; d| |G|} \Q(\xi_d)^{k_d }.$$
 \end{theorem}

One can also compute the primitive central idempotents of a rational group algebra of a finite abelian group. To do so, we introduce some notation.

Let $G$ be a finite group and $N$ a normal subgroup of $G$. Let $F$ be a field whose characteristic does not divide $|G|$. In $FG$ consider the elements
  $$\varepsilon (G,N)=\left\{  \begin{array}{ll}
                                          \widehat{G}  & \mbox{if } G=N\\
                                          \prod_{D/N\in M(G/N)} (\widehat{N}-\widehat{D}), & \mbox{otherwise}
                                      \end{array}\right.$$
Here $M(G/N)$ denotes the set of the minimal non-trivial normal subgroups $D/N$ of $G/N$, with $D$ a subgroup of $G$ containing $N$.
It easily is verified that $\varepsilon (G,N)$ is a central idempotent of $FG$.

\begin{lemma}
If $e$ is a primitive central idempotent of $\Q G$ such that $\Q G e$ is a field then
$e=\varepsilon (G,N)$ where $N=S_{G}(e)$ and $\Q Ge=\Q (\xi_d)$, where $d=|G/N|$.
\end{lemma}

\begin{corollary}
Let $G$ be a finite abelian group. The primitive central idempotents of $\Q G$ are the elements $\varepsilon (G,N)$ with $N$ a subgroup of $G$ such that $G/N$ is a cyclic group.
\end{corollary}

Note that primitive central idempotents of a complex group $\C G$ of a finite group also are well known.
Indeed, denote by $\mbox{Irr}(G)$ the set of the irreducible complex characters of $G$. If $\chi \in \mbox{Irr}(G)$ then 
 $$e (\chi) =\frac{\chi (1)}{|G|} \sum_{g\in G} \chi (g^{-1}) g$$
 is a primitive central idempotent of $\C G$. Moreover, it is the unique primitive central idempotent $e\in \C G$ such that $\chi(e) \neq 0$. 
 One can replace, in the above, the field $\C$ by any splitting field $F$ of $G$, that is, $F$ is a field such that $FG=\prod_i M_{n_i}(F)$.
 The Brauer splitting theorem states that $\Q (\xi_{|G|})$ is a splitting field of $G$ (where $\xi_{|G|}$ is a primitve $|G|$-th root of unity). More generally, it says that if $F$ is a field and $FG$ is semisimple
 then $F(\xi_{|G|})$ is a splitting field of $G$.
 
 If $FG$ is not necessarily split, then it much more complicated to describe the primitive central idempotents of $FG$. In theory one can determine the primitive central idempotents of $FG$, via Galois-descent, from the primitive central idempotents of $F(\xi_{|G|})G$. However, this does not necessarily result in
 some nice generic formulas.
Nevertheless, for some classes of groups one can obtain nice descriptions in terms of the subgroups of $G$. The class includes the abelian-by-supersolvable groups. We will explain such  formulas for $\Q G$.

We need to introduce some terminology and notation.

\begin{proposition}
Let $G$ be  a finite group and $H$ and $K$ subgroups of $G$ such that $K \subseteq H$. Then, $\text{Lin}(H,K)=\{ \chi \mid \chi \text{ a linear complex character with } \ker (\chi ) =K\}\neq \emptyset$ if and only if the following conditions hold
\begin{enumerate}
\item[(S1)] $K\lhd H$,
\item[(S2)] $H/K$ is cyclic.
\end{enumerate}
Assume that (S1) and (S2) hold for $\chi \in \text{Lin}(H,K)$. Then $\chi^{G}$ is absolutely irreducible if and only if $(H,K)$ satisfies the following condition:\\
(S3) for every $g\in G\setminus H$ there exists $h\in H$ such that $(h,g)\in H\setminus K$.

A Shoda pair of a finite group $G$ is a pair $(H,K)$ of subgroups of $G$ satisfying conditions (S1), (S2) and (S3).
\end{proposition}
\begin{proof}
The first part follows from the fact that every finite subbgroup of a field is cyclic. The second part is due to Shoda. 
\end{proof}

\begin{theorem} (Olivieri, del R\'io, Sim\'on)
If $(H,K)$ is a Shoda pair of a finite subgroup $G$ and $\chi \in \text{Lin}(H,K)$ then $\chi^{G}$ is an absolutely irreducible character and 
there is a unique primitve central idempotent $e$ of $\Q G$, denoted, $e_{\Q}(\chi)$, such that $\chi^{G}  (e)\neq 0$. Furthermore,
$$
e_{\Q}(\chi^{G}) =\frac{[\text{Cen}_{G}(\varepsilon(H,K)):H]}{[\Q (\chi):\Q(\chi^{G})]} e(G,H,K),
$$
where 
$$e(G,H,K) =\sum_{t\in T} \varepsilon (H,K)^{t}$$
and $T$ is a right transversal of $\text{Cen}_{G}(\varepsilon(H,K))$ in $G$.
The unique Wedderburn component containing $e(G,H,K)$ is $\Q G e(G,H,K)$, it will be   denoted $A_{\Q}(G,H,K)$.
\end{theorem}

A character of a finite group is said to be monomial  if it is the character afforded by a representation induced from a linear character.
One says that $G$ is a monomial group  if every irreducible complex character of $G$ is monomial.

\begin{corollary}
A finite group $G$ is monomial if and only if every primitive central idempotent of $\Q G$ is of the form $qe(G,H,K)$ for $(H,K)$ a Shoda pair of $G$ and $q\in \Q$.
\end{corollary}

This result allows to compute all primitive central idempotents of $\Q G$ for $G$ a finite monomial group and this without actually computing the monomial absolutely irreducible characters of $G$. It suffices to compute all the Shoda pairs $(H,K)$ of $G$, compute $e(G,H,K)$ and then compute the rational number $q$ such that $qe(G,H,K)$ is an idempotent. Note that different Shoda pairs can determine the same primitive central idempotent.
Janssens determined 
a formula for all primitive central idempotents of $\Q G$ for arbitrary finite groups $G$ (the main tool used  is Artin's Induction Theorem). Note that the formula yields a rational linear combination of elements of the form $e(G,C,C)$ where $C$ is a cyclic subgroup of $G$; but in general $e(G,C,C)$ is not an idempotent.

So, for some classes of groups one can compute explicitly the primitive central idempotents $e$. A next step is to determine a description of the simple component $\Q Ge$ for a given central idempotent $e(G,H,K)$. In order to compute the unit group $\U (\Z G)$ one would like to obtain a concrete description that yields control on the rational representations (without having to calculate the character table of $G$).
We now state how this can be done  for $e(G,H,K)$ provided the Shoda Pair  satisfies some additional  conditions. 
All these results are due to Olivieri, del R\'{i}o and Sim\'on.

A useful lemma is the following.

\begin{lemma}
Let $H$ and $K$ be subgroups of a finite group $G$ such that $K\lhd H$ and $H/K$ is cyclic. Assume $\varepsilon (H,K) \varepsilon (H,K)^{g}=0$ for all $g\in G\setminus \text{Cen}_{G}(\varepsilon(H,K))$. Then $\text{Cen}_{G}(\varepsilon (H,K))=N_{G}(K)=\{ g\in G \mid g^{-1}Kg=K\}$.
\end{lemma}

\begin{proposition}
Let $G$ be a finite group and let $H$ and $K$ be subgroups such that $K\subseteq H$. The following conditions are equivalent.
\begin{enumerate}
\item $(H,K)$ is a Shoda pair of $G$, $H\lhd N_{G}(K)$ and the different $G$-conjugates of $\varepsilon (H,K)$ are orthognal.
\item $(H,K)$ is a strong Shoda pair, that is,
\begin{enumerate}
\item[(SS1)] $H\lhd N_G(K)$,
\item[(SS2)] $H/K$ is cyclic and maximal abelian subgroup of $N_G(K)/K$ and
\item[(SS3)] for every $g\in G\setminus N_G(K)$, $\varepsilon (H,K) \varepsilon(H,K)^{g}=0$
\end{enumerate}
\item The following conditions hold
\begin{enumerate}
\item[(SS1')] $H\lhd \text{Cen}(\varepsilon (H,K))$,
\item[(SS2')] $H/K$ is cyclic and a maximal abelian subgroup of $\text{Cen}(\varepsilon (H,K))$ and
\item[(SS3')] for every $g\in G\setminus \text{Cen}(\varepsilon (H,K))$,  $\varepsilon (H,K) \varepsilon(H,K)^{g}=0$.
\end{enumerate}
\end{enumerate}
A finite group is said to be strongly monomial if every irreducible complex character  of $G$ is strongly monomial, i.e. it is of the form $\chi^{G}$ for $\chi \in \text{Lin}(H,K)$ and $(H,K)$ a strong Shoda pair of $G$. Note that for such a group every primitive central idempotent of $\Q G$ is of the form $e(G,H,K)$ with $(H,K)$ a strong Shoda pair of $G$.
\end{proposition}

\begin{theorem}
Every abelian-by-supersolvable finite group is strongly monomial.
\end{theorem}

A useful fact to prove this result is the following. If $G$ is finite supersolvable group and $N$ a maximal abelian normal subgroup of $G$ then $N$ is a maximal abelian subgroup of $G$. 

\begin{proposition}
Let $(H,K)$ be a pair of subgroups of a finite group $G$ such that $K\lhd H\lhd G$ and satisfying (SS2) (i.e. $H/K$ is cyclic and a maximal abelian subgroup of $N_G(K)/K$). Then $(H,K)$ is a strong Shoda pair of $G$.
\end{proposition}

\begin{theorem}
Let $G$ be a finite metabelian group and let $A$ be a maximal abelian subgroup of $G$ containing the commutator subgroup $G'$. The primitive central idempotents of $\Q G$ are the elements of the form $e(G,H,K)$ where $(H,K)$ is a  pair of subgroups of $G$ satisfying the following conditions:
\begin{enumerate}
\item $H$ is a maximal element in the set $\{ B\leq G \mid A\leq B \text{ and } B'\subseteq K \subseteq B\}$ and
\item $H/K$ is cyclic.
\end{enumerate}
\end{theorem}

For a finite group $G$ one  can describe the simple component of $\Q G$ associated to a strong Shoda Pair.

\begin{theorem}
Let $(H,K)$ be a strong Shoda pair of the finite group $G$ and $\chi \in \text{Lin}(H,K)$, $N=N_G(K)$, $n=[G:N]$, $h=[H:K]$ and $\overline{x}=xK$ a generator of the group $H/K$. The following properties hold.
\begin{enumerate}
\item $N=\text{Cen}_G(\varepsilon (H,K))$,
\item $e_{\Q}(\chi^{G}) =e(G,H,K)$,
\item The mapping $\sigma : N/H\rightarrow \text{Gal} (\Q (\xi_h )/\Q (\chi^{G}))$ defined by $\overline{y}\mapsto \sigma_{\overline{y}}$, for $\overline{y}\in N/H$, with 
 $$\sigma_{\overline{y}} (\xi_h ) =\xi_{h}^i$$
 where $i$ is such that $\overline{y} \overline{x} \overline{y}^{-1} =\overline{x}^{i}$, is an isomorphism.
 \item $A_{\Q}(G,H,K) \cong M_{n} \left(\Q (\xi_{h}) * (N/H)\right) \cong M_{n}(\Q (\xi_h)/\Q (\chi^{G}),f)$, where $f$ is the element of $H^{2}(N/H, H/K)$ associated to the extension
  $$1\rightarrow H/K \stackrel{\chi}{\cong} \langle \xi_h \rangle \rightarrow N/K \rightarrow N/H \stackrel{\sigma}{\cong} \text{Gal} (\Q (\xi_{h})/\Q (\chi^{G})) \rightarrow 1 .$$
More precisely, for every $a\in N/H$ fix a preimage $u_a$ of $a \in N/K$. Then,
 $$f(a,b) =\xi_h^j,$$
 where $j$ is such that $u_a u_b =\overline{x}^{j} u_{ab}$.
 More explicit,
 choose a right transversal $T$ of $H$ in $N$. Then 
 $$\Q (\xi_h ) * (N/H) =\sum_{t\in T} \Q (\xi_h) u_t$$
 The action $\alpha :N/H \rightarrow \text{Aut}(\Q (\xi_h ))$ is defined in part (3) as follows.
 For $\overline{y} \in N/H$, $\alpha_{\overline{y}} =\sigma_{\overline{y}}$. The twisting $f: N/H \times N/H \rightarrow \U (\Q (\xi_h ))$ is defined by $f(\overline{x}, \overline{y}) =\xi_{h}^{j}$ if  $t_x t_y =x^{j} k_{xy}t_{xy}$ with $t_x ,t_y\in T$ so that $t_xH=\overline{x}$, $t_y H=\overline{y}$, $k_{xy}\in K$ and $j\in \Z$.
 \item Let $F$ be a field of characteristic zero and let $G_F =\text{Gal}(F(\chi )/F(\chi^{G}))$. Consider $G_F$ as a subgroup of $G_{\Q}$ via the restriction
 $G_F \rightarrow G_{\Q}$. Then
  $$A_{F}(\chi^{G}) =M_{nd} (F(\xi_{h})/F(\chi^{G}), f'),$$
where $d=\frac{[\Q (\xi_h):\Q(\chi^{G})]}{[F(\xi_h):F(\chi^{G})]}$ and $f'(\sigma ,\tau ) = f(\sigma|_{\Q (\xi_h )}, \tau|_{\Q (\xi_h)})$ (and this is the unique simple component  $FGe$ with $\chi^{G}(e)\neq 0$).
\end{enumerate}
\end{theorem}

In \cite{BakshiKaur2} Bakshi and  Kaur  introduced the class of generalized strongly monomial groups. This is based on
generalized strong Shoda pairs of a finite group and leads to the class  generalized strongly monomial groups.
In addition to strongly monomial groups, the class  of generalized strongly monomial groups  also contains all subnormally monomial groups and, more generally,  the class of finite groups $G$ such that all subgroups and quotient groups of subgroups of $G$ satisfy the following property: either they abelian or they contain a noncentral abelian normal subgroup.  For the class of generalized strongly monomial finite groups the primitive central idempotents of its rational group algebra are described  as well as the corresponding  simple component associated to each generalized strong Shoda pair of $G$.
For other recent work on the topic of describing primitive central idempotents we refer to \cite{BakshiKaur1,BPK,BM1,BM2}.
For references on  applications of the description of primitive central idempotents and their corresponding simple components we refer the reader to  \cite{bookvol1,bookvol2}.

\section{ Rational Wedderburn decomposition}

Let $F$ be a field of characteristic different from $2$. Recall that an $F$-algebra $A$ is said to be a quaternion algebra over $F$ if there exists $a,b\in \U (F)$ such that
   $$A= 
		\quat{a,b}{F} = \frac{F\langle i, j\rangle}{(i^2=a,\ j^2=b,\ ij=-ji)} =F1 +Fi +Fj +Fk,$$
where $k=ij$.
Recall the norm map $N: \quat{a, b}{F}\rightarrow F$, defined by
  $x=x_0 +x_1 i + x_2 j + x_3 k \mapsto x\overline{x}$, where $\overline{x} =x_0 -x_1 i -x_2 j -x_3 k$,
(with $x_0, x_1, x_2,x_3 \in F$). The map $x\mapsto \overline{x}$ defines an involution on $\quat{a,b}{F} $, called the quaternion conjugation.

Note that $A=\quat{a, b}{F}$ is a simple algebra with center $F$ (i.e. it is a central simple $F$-algebra) and thus it is either a division algebra or it is isomorphic to $\text{M}_{2}(F)$. The following conditions are equivalent:
\begin{enumerate}
\item $A=\text{M}_2(F)$,
\item $N(x)=0$ for some $0\neq x\in A$,
\item $u^2=av^2+bw^2$ for some $0\neq (u,v,w)\in F^3$.
\end{enumerate}

\begin{definition} A simple finite dimensional rational algebra is said to be exceptional if it is one of the following types:
\begin{enumerate}
 \item[type 1:]  a non-commutative division algebra  other then a totally definite quaternion algebra $\quat{a,b}{F} $ over a number field $F$, that is, $F$ is totally real and $a,b<0$.
 \item[type 2:] a $2\times 2$-matrix ring over the rationals, a quadratic imaginary extension of the rationals or over a totally definite quaternion algebra over $\Q$.
\end{enumerate}
\end{definition}

Amitsur described the finite subgroups  that are contained in an exceptional simple component of type 1.
Unit groups of orders in such division algebras are a big unknown. The reader is referred to Kleinert's book on this topic \cite{KleinertBook}.
Note that, because of Dirichlet's unit theorem (see below)  and  a result of Kleinert, the exceptional simple components of type 2 are precisely those $\text{M}_2(D)$ for which an order $\mathcal{O}$ in $D$ has only finitely many units. 
Further,  all finite dimensional rational   non-commutative division algebras are of type 1 except those  for which the unit group of an order has a central subgroup  of finite index.

For a field $F$ and a finite dimensional semisimple rational algebra $A$, we denote 
by $r_{F}(A)$ the number  of simple Wedderburn components of $F\otimes_{\Q} A$.

\begin{theorem}
(Dirichlet's Unit Theorem) Let $F$ be a number field and assume that $F$ has $r$ real embeddings and $s$ pairs of complex non-real embedding. If $R$ is the ring of integers of $F$ then
 $$\U (R) =T\times A,$$
 where $T$ is a finite group formed by roots  of units in $F$ and $A$ is a free abelian group of rank $r+s-1$. 
 Note that this rank equals $r_{\R}(F)-r_{\Q}(F)$ and $F\otimes_{\Q}\R \cong \R^r \times \C^{s}$.
\end{theorem}

We recall some notions concerning the rational group algebra $\Q G$. Let $e_1,\dots,e_n$ be the primitive central idempotents of $\Q G$, then $$\Q G=\Q Ge_1\oplus \dots\oplus \Q G e_n,$$ where each $\Q G e_i$ is identified with the matrix ring $M_{n_i}(D_i)$ for some division algebra $D_i$. For every $i$, let   $\mathcal{O}_i$ be an order in $D_i$. Then $M_{n_i}(\mathcal{O}_i)$ is an order in $\Q Ge_i$. Denote by $\GL_{n_i}(\mathcal{O}_i)$ the group of invertible matrices in $M_{n_i}(\mathcal{O}_i)$.

Let $\mathcal{O}$ be an order in a finite dimensional rational division algebra $D$. 
Then
 $$\SL_n (\mathcal{O}) = \{ x\in \GL_n (\mathcal{O}) : \text{nr}(x)=1\},$$
where $nr$ is the reduced norm,
 and for  subset $I$ in $\mathcal{O}$ we put
    $$E(I)=\langle I+ xE_{lm} \mid  x\in I,\ 1\leq l,m\leq n_i,\ l\neq m,\ E_{lm} \textrm{ a matrix unit} \rangle  \subseteq \SL_n(\mathcal{O}).$$

\begin{theorem} (Bass-Vaser{\v{s}}te{\u\i}n-Liehl-Venkataramana) \label{elementary} \label{BassVaserstein} 
Let $\mathcal{O}$ be an order in a finite dimensional rational division algebra $D$. 
Assume that $n$ is an integer and $n\geq 2$. If the simple algebra $\text{M}_{n}(D)$ is not exceptional then
 $[\SL_{n}(\mathcal{O}):E(I)]<\infty$ for any non-zero ideal $I$ of $\mathcal{O}$.
\end{theorem}

In this section we restrict the type of $2\times 2$-matrices which can occur as simple components in the Wedderburn decomposition of $\Q G$ for finite groups $G$. We also give a classification of those finite groups which have a faithful exceptional $2\times 2$-matrix ring component (i.e. $G$ embeds naturally into the simple component).

Surprisingly,  if one assumes $\text{M}_2(D)$ to be an exceptional component of $\mathbb{Q}G$, then the possible  parameters $d$ (resp.  $(a,b)$) of   
$D= \Q(\sqrt{-d})$ (resp.\ $\left( \frac{a,b}{\Q}\right)$) are very limited. It was proven by Eisele, Kiefer and Van Gelder  
\cite{EKVG} 
that only a finite number of division algebras can occur and, moreover, the possible parameters have been  described.
Together with the results of B\"achle,  Janssens, Jespers,  Kiefer and  D. Temmerman in \cite{BJJKT}
one has the following result.

\begin{theorem}\label{possible exceptional components}
Let $G$ be a finite group and $e$ a primitive central idempotent of $\Q G$ such that $\Q G e$ is exceptional. Then
\begin{enumerate}
\item if $\Q Ge$ is of type 2 over a field $\Q(\sqrt{-d})$, then $d \in \{ 0, -1, -2 , -3 \}$, 
\item if $\Q Ge$ is of type 2 over a quaternion algebra $\left( \frac{a,b}{\Q}\right)$,\\ then $(a,b) \in \{ (-1,-1),( -1, -3), (-2,-5) \}$, 
\item if $G$ is cut, i.e. all central units are trivial,  and $\Q Ge \cong \text{M}_2\left( \frac{-1,-3}{\Q}\right)$ or $\Q Ge \cong \text{M}_2(\Q(\sqrt{-2}))$ 
then there exists another primitive central idempotent $e'$ such that $\Q G e' \cong \text{M}_2(\Q)$ or $\Q Ge' \cong \text{M}_2(\Q(i))$,
 \item there exists a primitive central idempotent $e$ of $\mathbb{Q}G$ such that $\mathbb{Q}Ge \cong \text{M}_2\left( \frac{-2,-5}{\Q}\right)$ if and only if $G$ maps onto $G_{240,90}$, (we refer to  the Small Groups Library of GAP
 for the definition of $G_{240,90}$),
\item if $G$ is solvable and cut, then $\Q Ge \ncong \text{M}_2 \left( \frac{-2,-5}{\Q}\right)$, 
\item if $G$ is cut, then $\Q G e$ cannot be of type 1. 
\end{enumerate}
\end{theorem}

In the above theorem, also the groups $Ge$ that yield an exceptional simple component of type 2 can be described; there are less than 60 such groups.

All the fields and division algebras appearing in the previous theorem have the peculiar property to contain a Euclidean order $\mathcal{O}$ which 
therefore is maximal and unique up to conjugation.
This yields that also 
all the $2\times 2$-matrix algebras in the Theorem
have, up to conjugation, a unique maximal order, namely $\text{M}_2(\mathcal{O})$. Recall that in case of $\Q(\sqrt{-d})$, with $ d \in \{ 0,1,2,3 \}$, the unique maximal order is their respective ring of integers $\I_d$ and in case of $\mathbb{H}_2, \mathbb{H}_3,\mathbb{H}_5$ the respective maximal orders can easily be described;
where \[\mathbb{H}_2 = \left( \frac{-1,-1}{\Q}\right),\quad \mathbb{H}_3 =
\left( \frac{-1,-3}{\Q}\right)
\quad \mbox{ and }\quad \mathbb{H}_5 = 
\left( \frac{-2,-5}{\Q}\right) . \] 
Recall that   a domain $R$ 
is said to be  a \emph{left Euclidean} ring if  there exists a map  $\delta$ from  $R\setminus \{ 0\}$ to the non-negative integers such that 
$$\forall\ a,b \in R \text{ with } b\neq 0, \exists\  q,r \in R: a= qb + r\text{ with } \delta(r) < \delta(b) \text{ or } r = 0;$$ 
and  $R$ is said to be a \emph{right Euclidean} ring 
if  there exists a map  $\delta$ from  $R \setminus \{0\}$ to the non-negative integers   such that 
 $$\forall\ a,b \in R \text{ with } b\neq 0, \exists\  q,r \in R: a= bq + r\text{ with } \delta(r) < \delta(b) \text{ or } r = 0.$$

\section{ Generators for a subgroup of finite index}

Let $G$ be a finite group. We know that $\Z G$ is an order in $\Q G$ and that $\Z G$ only has trivial idempotents.

\begin{lemma}
 Let $K$ be a field extension of $\Q$ and let $e = \sum_{g\in G} e_g g \in KG$, with each $e_g\in K$.
 If $e^2 = e\not\in  \{0,1\}$ then $e_1$ is a rational number in the interval $(0, 1)$.
 \end{lemma}

Now if $e_1 , \ldots , e_n$ are the primitive central idempotents of $\Q G$ then also $\sum_{i=1}^{n} \Z G e_i$ is an order in $\Q G$ that contains $\Z G$.
Their unit groups, however, do not differ a lot in size.
Indeed we have the following properties.

\begin{lemma}\label{ordersunits}
Let $A$ be a semisimple  finite dimensional rational algebra. Let $e_1, \ldots , e_n$ be the primitive central idempotents of $A$.
\begin{enumerate}
\item Every element of an order $\mathcal{O}$  in $A$ is integral over $\Z$.
\item The intersection of two orders of $A$ is again an order in $A$.
\item Every order of $A$ is contained in a maximal order of $A$, say $\mathcal{M}$. Furthermore, $\mathcal{M} =\sum_{i=1}^{n} \mathcal{M}e_i$ and each $\mathcal{M}e_i$ is a maximal order in $Ae_i$.
\item Suppose $\mathcal{O}_{1} \subseteq \mathcal{O}_{2}$ are two orders in $A$. Then
 \begin{enumerate}
  \item $u\in \mathcal{O}_{1}$ is invertible in $\mathcal{O}_{2}$ if $u^{-1}\in \mathcal{O}_{1}$.
  \item the index of the unit groups $(\U (\mathcal{O}_{2}) : \U (\mathcal{O}_{1}))$ is finite.
 \end{enumerate}
\end{enumerate}
\end{lemma}
\begin{proof}
We only prove part (4). 

(a) Let $u\in \mathcal{O}_{1}$ and assume $u^{-1}\in \mathcal{O}_{2}$. Using indices of additive subgroups, we get
$[\mathcal{O}_{2}:u\mathcal{O}_{1}] =[u\mathcal{O}_{2}:u\mathcal{O}_{1}] \leq [\mathcal{O}_{2}:\mathcal{O}_{1}]$. Hence, $u\mathcal{O}_{2}=\mathcal{O}_{1}$ and thus $u$ is invertible in $\mathcal{O}_{1}$. The converse is obvious.

(b) Since $\mathcal{O}_{2}$ is a free  $\Z$-module containing $\mathcal{O}_{1}$, they both have equal $\Z$-rank, say $n$. Thus the index of the addtive groups satisfies $[\mathcal{O}_{2}:\mathcal{O}_{1}]=m< \infty$. Hence,
$m\mathcal{O}_{2}\subseteq \mathcal{O}_{1}$. Suppose now that $u,v\in \U(\mathcal{O}_{2})$ such that $u+m\mathcal{O}_{2}=v+m\mathcal{O}_{2}$. Then $u^{-1} v-1 \in m\mathcal{O}_{2}\subseteq \mathcal{O}_{1}$ and thus $u^{-1} v\in \mathcal{O}_{1}$. Similarly, $v^{-1}u\in \mathcal{O}_{1}$. So, $u^{-1}v\in \U(\mathcal{O}_{1})$.
Hence, we have shown that $(\U (\mathcal{O}_{2}): \U(\mathcal{O}_{1}) \subseteq [\mathcal{O}_{2}:m\mathcal{O}_{2}]<\infty$.
\end{proof}

Hence, to compute a subgroup of finite index in $\U (\Z G)$ it is sufficient to construct for each primitive central idempotent $e_i$ of $\Q G$  units $u$ of $\U (\Z G)$, that belong to $(1-e_i)+\Z G e_i$, and they are such that all the units $ue_i$  generate a subgroup of finite index in $\U (\Z G e_i)$.
The next proposition shows that for the latter we have to describe units that contribute both  to a large subgroup of the center of $\U (\Z G e_i)$ and to a large subgroup of the units of reduced norm one in $\Z Ge_i$.

\begin{proposition}
Let $\mathcal{O}$ be an order in a simple finite dimensional rational algebra $A$. Then $\GL_{n}(\mathcal{O})$ contains a subgroup of finite index which is 
isomorphic to a subgroup of finite index in $\SL_{n}(\mathcal{O}) \times \U (R)$, where $R$ is the unique maximal order in the center of $A$.
\end{proposition}

Let us now focus on the units of reduced norm one.
For this a crucial and well known lemma is the following.

\begin{lemma}\label{MatrixUnitsL}
Let $D$ be a finite dimensional rational division algebra and let $n$ be an integer with $n>1$. If $f$ is a non-central idempotent in $M_n(D)$ then there exist matrix units $E_{i,j}$, with $1\leq i,j\leq n$ (that is, $\sum_{i=1}^{n}E_{i,i} =1$ and $E_{i,j}E_{k,l}=\delta_{j,k}E_{i,l}$) such that 
  $$f=E_{1,1} + \cdots + E_{l.l},$$
with $0<l<n$. Moreover, $M_n(D)=M_n(D')$,  with $D'$ the centraliser of all $E_{i,j}$.

\end{lemma}

One can then prove the following result.  We first introduce some notation.
Let $A$ be a semisimple finite dimensional rational algebra such that $AG$ is semisimple. Let $R$ be an order in $A$ and let $x_1, \dots , x_m$ be a generating set  of $R$ as an $\Z$-module.
For a given set of idempotents $\mathcal{F}$ of $AG$ we put
  $$\text{GBic}^{\mathcal{F}}(RG) =\langle b(x_i g, f),\; b(f,x_i g) \mid f\in \mathcal{F}, \; g\in G, \; 1\leq i\leq m\rangle  .$$
If $R=\Z$ then we simply put
  $$\text{GBic}^{\mathcal{F}} (G).$$
If, furthermore,  $\mathcal{F}=\{ \widehat{g} \mid g\in G\}$ then we put
 $$\text{Bic}(G)$$
for this group.

\begin{theorem} (Jespers-Leal)
Let $G$ be a finite group and $R$ an order in a semisimple finite dimensional algebra $A$. Assume $AG$ is semisimple, $e$ is a primitive central idempotent of $AG$ and $\mathcal{O}$ is an order in $AGe$. Assume the simple component $AGe$ is not exceptional. If $f$ is an idempotent of $AG$ such that $ef$ is non-central (in $AGe$) then $\text{GBic}^{\{e\}}(RG)$ contains a subgroup of finite index in the reduced norm one elements of $1-e+\mathcal{O}$.
\end{theorem}

\begin{proof} 
Let  $n_{f}$ be the minimal positive integer such that $n_{f}f\in R G$.
Let $x_{1}, \ldots , x_{m}$ be a generating set of $R$ as a $\Z$-module.
As $AG e =M_{n}(D)$, for some division algebra $D$,
by Lemma~\ref{MatrixUnitsL} there is
a set of matrix units $\{ E_{i,j}  :  1\leq i,j \leq n\}$ of $AG e$ with $f=E_{1,1}+\dots+E_{l,l}$  for some $0<l<n$.
Recall from Lemma~\ref{ordersunits} that the unit groups of two orders in $AG e$ are commensurable.
Hence, without loss of generality, we may assume that  $M_{n}(\mathcal{O} )$ is the order chosen in the statement, with $\mathcal{O}$ an order in $D$.
Let $J=\text{GBic}^{\{ f\}} (R G)$. Note that
  $$ \left[  1 + n_{f}^{2}fx_{i}g(1-f) \right]^{k}
     \left[  1 + n_{f}^{2} f x_{j}h (1-f) \right]^{l}=
     \left[  1 + n_{f}^{2} f(kx_{i}g+ lx_{j}h) (1-f) \right],
  $$
for every $k,l \in \Z$, $g,h \in G$ and $1\leq i,j\leq m$. So, the group generated by
these units contains all elements of the form
   $$ 1+n_{f}^{2} f \alpha (1-f), \quad \hbox{\rm and} \quad 1+n_{f}^{2} (1-f) \alpha f,$$
with $\alpha \in R G.$

Since
  $$ \left\{ 1+n_{f}^{2}f \alpha (1-f), \; 1+n_{f}^{2} (1-f) \alpha f : \alpha \in R G \right\} \subseteq J, $$
it follows that
  $$ \left\{ 1+n_{e}n_{f}^{2} f \alpha (1-f) e, \; 1+n_{e}n_{f}^{2} (1-f) \alpha f e : \alpha \in R G
    \right\} \subseteq J .$$

Let $i\leq l$ and  $l+1 \geq j \geq n$. Then,
 $$f \mathcal{O} E_{i,j}(1-f)e = \mathcal{O} E_{i,j}.$$
Hence, as ${\cal O}$ is a finitely generated $\Z$-module,
there exists a positive integer $n_{i,j}$ such that
$$1+n_{i,j} \mathcal{O}  E_{i,j} \subseteq J \cap \SL_n(\mathcal{O}).$$
And similarly,
  $$ 1+n_{j,i} \mathcal{O} E_{j,i} \subseteq J \cap \SL_n(\mathcal{O}), $$
for some positive integer $n_{j,i}$.

So we have shown the existence of a positive integer $x$ with
  $$ 1+x\mathcal{O} E_{i,j} \subseteq J \cap \SL_n(\mathcal{O}) \quad \hbox{\rm and} \quad
     1+x\mathcal{O} E_{j,i} \subseteq J \cap \SL_n(\mathcal{O}),$$
for all $1\leq i \leq l$ and $l+1\leq  j \leq n$.

Now let $1\leq i,j \leq l,$ $i\neq j$ and $\alpha \in \mathcal{O}.$
Then one easily verifies that
  $$ 1+x^{2}\alpha E_{i,j} = ( 1+x\alpha E_{i,l+1},\; 1+xE_{l+1,j}) \in J \cap \SL_n(\mathcal{O}).$$

Similarly, for $l+1\leq i,j \leq n_{i}$, $i\neq j$, it follows that
  $$ 1+x^{2}\mathcal{O} E_{i,j} \subseteq J \cap \SL_n(\mathcal{O}). $$
Because of the assumptions, the  result now follows from Theorem~\ref{BassVaserstein}.
\end{proof}

The next step is to construct in a simple component $\Q Ge$ a non-central idempotent. This can be done if $Ge$ is not fixed point free and one can show that this can be done with an idempotent of the type $\widehat{g} e$.  Recall that a finite group is said to be {\it fixed point free} if it has an (irreducible) complex representation $\rho$ such that $1$ is not an eigenvalue of $\rho (g)$ for all $1\neq g\in G$. Such groups show up naturally, as every non-trivial finite subgroup of a division algebra is fixed point free.

Indeed, Let $e$ be  a primitive central  idempotent of $\Q (\xi )G$ with $Ge$ not commutative and 
 $Ge$  not fixed point free.
Thus, there exists a primitive central idempotent $e_{1}$ of $\C G e$
such that the non-linear
 complex representation
$\rho : G \rightarrow (\C G)e_{1}$ mapping $x$ onto $xe_{1}$ has eigenvalue $1$ for some $\rho (g)$, with $g\in G$ and $ge_{1}\neq e_{1}$.
Since $\rho (g)$ is
diagonalizable one  may assume that
$$\rho(g) = \pmatriz{{cc} I_j & 0 \\ 0 & D} \text{ with } 1\le j < n \text{ and } D=\diag(\xi_{j+1},\dots,\xi_n)$$
and $\xi_{j+1},\dots,\xi_n$ are roots of unity different from $1$.
Consequently
$$\rho ({\widehat{g} } ) = \pmatriz{{cc} I_j & 0 \\ 0 & 0}.$$
Hence $\widehat{g} e_{1}$ is a  non-central idempotent of $\C G$.
It follows that  $\widehat{g}e$ is a non-zero idempotent in $\Q (\xi )
Ge.$
Furthermore $\widehat{g} e\neq e$, because otherwise
$\widehat{g} e_{1} = \widehat{g} e_{1} e = e_{1}e =e_{1}$, a contradiction.

Now it remains to find units that cover the center of $\U (\Z G)$. This is done via the following beautiful result of Bass-Milnor. 
Let $\mathcal{O}$ be an order in a semisimple finite dimensional rational algebra $A$ and let $G$ be a finite group. Then  the natural images  of the units of $Z(\mathcal{O})C$, 
where  $C$ runs through the  cyclic subgroups of $G$, give a subgroup of finite index in $K_{1}(Z(\mathcal{O}G))$.  Now another beautiful result of Bass-Milnor says that 
the Bass units $u_{k,m}(\xi^{i} g)$ generate a subgroup of finite index in $\U(\Z [\xi] \langle g \rangle)$.
One knows even specific Bass units that are a basis of free abelian subgroup of finite index..

All the above mentioned results then give the following result.

\begin{theorem}
Let $G$ be a finite group and $\xi$ a root of unity. Suppose that $\Q (\xi ) G$ does not have exceptional simple components. Let $\mathcal{C}=\{ \widehat{g} \mid g\in G\}$.
Suppose that for every primitive central idempotent $e$ of $\Q G$ the group $Ge$ is not fixed point free. Then
 $$\langle \text{GBic}^{\mathcal{C}} (\Z [\xi ] G) \cup \text{Bass }(\Z [\xi ] G) \rangle$$
 is of finite index in $\U (\Z[ \xi ]G)$.
\end{theorem}

The result also implies that the unit group is finitely generated.
One has a much stronger result due to Siegel.

\begin{theorem}
Let $\mathcal{O} $ be an order in a finite dimensional semisimple rational algebra $A$. Then $\U (\mathcal{O})$ is finitely presented.
\end{theorem}

We give some examples of finite $2$-groups $G$ such that the Bass units together with the bicyclic units do not generate a subgroup 
of finite index in $\U (\Z G)$.
The following result is due to Jespers and Parmenter.

\begin{theorem}
Let $D_8 =\langle a,b \mid a^4=1, \; b^2=1,\; ba=a^3 b \rangle$, the quaternion group of order $8$. Let $G$ be a finite $2$-group and suppose there exists an epimorphism  $f:G \rightarrow D_8$. If at least 
two of the elements $b,ab,a^2b,a^3b$ do not have preimages in $G$ of order $2$, then the Bass units together with the bicyclic units in $\Z G$ do not generate a subgroup of finite index in $\U (\Z G)$.

In particular, this applies to the groups $Q_{16}$, $\langle a,b \mod a^8=1, \; b^2=1, ba=a^3b\rangle$, $C_4 \rtimes C_2$ and $\left( \langle z \rangle_2 \times \langle a \rangle_4\right) \rtimes \langle b\rangle_2$, with $z$ central and $a^b=za$.
\end{theorem}
\begin{proof}
The $\Z$-linear extension of $f$  to  a ring epimorphism $\Z G\rightarrow \Z D_{8}$, as well as the induced group
homomorphism $ \U (\Z G)\rightarrow \U (\Z D_{8})$, we also denote by $f$.

Since every Bass unit of $\Z D_{8}$ belongs to $D_8$, every Bass unit in $\Z G$ must map to an element of $D_{8}$.

Next consider a bicyclic unit $b(g,\widetilde{h})$ in $\Z G$.
Then either $f(b(g,\widetilde{h}))=1$ or
  $$f(b(g,\widetilde{h}) )= 1 +c(1-f(g))f(h)\widetilde{f(g)} = (1+(1-f(g))f(h)\widetilde{f(g)})^{c} ,$$
where $c=\frac{o(g)}{o(f(g))}$.

The bicyclic units of $\Z D_{8}$ are $u_{1}=b(a,\widetilde{b})$,
$u_{2}=b(a,\widetilde{ab})$, $u_{3}=b(a,\widetilde{a^2b})$ and $u_{4}=b(a,\widetilde{a^3 b})$.
Further $u_{4}=u_{3}^{-1}u_{2}^{-1}u_{1}^{-1}$. It is easily verified that
the given condition on $G$ yields that at least two of these bicyclic
units are not images of bicyclic units in $\Z G$.

It is known that
   $$V=\U (\Z D_{8}) \cap (1+\ker(\text{aug} ) (1-a^{2})) =\U (\Z D_{8}) \cap (1+\ker(\text{aug} ) (1-
       a))$$
is a normal complement
of the trivial units $\pm D_8$ and it is a free group of rank three, generated by the bicyclic units of the type $b(g,\widetilde{h})$.
Let $B$ be the subgroup of $\U (\Z G)$ generated by the Bass
units  and the bicyclic units of the type $b(g,\widetilde{h})$. Since $G$ is a $2$-group, it follows from the remarks above that $f(B)$ is a proper subgroup of
$V$ requiring at most $4$ generators.
Since $V$ is a free group of rank $3$, we conclude that $f(B)$ must be of infinite index in $V$.
Indeed, by the Nielsen-Schreier Theorem, if $f(B)$ has index $n$ in $V$ then $f(B)$ is free of rank $2n+1$.
As $f(B)$ is generated by at most $4$ elements, necessarily $n=1$ and hence $f(B)=V$, a contradiction.

For a positive integer $i$,  let $V_{i}$ denote the subgroup of $V$ consisting of those
units which can be written in the form $1+2^{i}\beta (1-a^{2})$ for some
$\beta \in \Z D_{8}$.
Because  $(1-a^2)^2=2(1-a^2)$, it follows that each $V_{i}\subseteq V$ .
Also note that for all $i$,  $V_{i}$ is a normal subgroup of $V$ and that
the groups      $V/V_{1}$ and $V_{i}/V_{i+1}$  are of exponent $2$ and thus abelian.
Since $\U (\Z D_8)$ is finitely generated, so is the group $V$.
Consequently, $V/V_{1}$ and all $V_{i}/V_{i+1}$ are finite.
So, each  $V/V_{i}$ is finite.

Let $K=\ker(f)$.
Obviously, $\lvert K\rvert =2^{l}$ for some $l\geq 1$.
We claim that $V_{l}\subseteq f(\U (\Z G))$.
Indeed, let $1+2^{l}\beta (1-a^{2})\in V_{l}$.
Choose $a_{1},\beta_{1}\in \Z G$ such that $f(a_1) =a$ and $f(\beta_1)=\beta$.
Put $u=1+\widetilde{K} \beta_{1} (1-a_{1}^{2})$.
Clearly $u\widehat{K} =\widehat{K}(1 +2^{l}\beta_1 (1-a_1^{2}))$ is a unit in $\Z G \widehat{K}\cong \Z D_{8}$.
Since  $u(1-\widehat{K}) =1-\widehat{K} $
is a unit in $\Z G (1-\widehat{K})$, we get that
$u\in \Z G$ is a unit in the order $\Z G \widehat{K} \oplus \Z G (1-\widehat{K})$.
Hence, because of Lemma~\ref{ordersunits},
$u\in \U (\Z G)$.
Obviously, $f(u)=1+2^{l}\beta (1-a^{2})$.
So, $u\in f(\U (\Z G))$ and the claim has been proved.

Suppose that $f(B)$ is of finite index in $f(\U (\Z G))$.
Since $f(B) \subseteq V$, this yields $f(B)$ is of finite index in $f(\U (\Z G))\cap V$.
Because $V_{l}\subseteq f(\U (\Z G))V$ and $V_l$ is of finite index in $V$, it follows that  $f(B)$ is of finite index in  $V$.
However this contradicts the earlier fact that  $f(B)$ is of infinite index in $V$.
Therefore, we have shown that $f(B)$ is of infinite index in $f(\U (\Z G))$.

To finish the proof we note that if $f(b(g,\widetilde{h}))\neq 1$ then it is a power of a bicyclic unit 
$b(\widetilde{f(h)},f(g))=1+(1+f(h))f(g)(1-f(h))$.
Since $b(\widetilde{f(h)},f(g))=b(f(g),\widetilde{a^{2}f(h)})$, we obtain that 
$f(\text{Bix} (G)) =f \left( \langle  b(g,\widetilde{h})\mid g,h\in G \rangle \rangle \right)$.
So, from the previous,
$f\left (\langle \text{Bix} (G) \cup \text{Bass} (G)\rangle \right)$ is of infinite index in $f(\U (\Z G))$
and thus $\langle \text{Bix} (G) \cup \text{Bass} (G)\rangle$ is of infinite index in $\U (\Z G)$.
\end{proof}

\section{Constructions of central units from Bass units}

We know that the central units of an order in a finite dimensional rational algebra form a finitely generated group.
As a consequence of Dirichlet's Unit Theorem and Lemma~\ref{ordersunits}, the rank of this group  also can be determined.

\begin{theorem}
Let $A$ be a finite dimensional semisimple rational algebra and $\mathcal{O}$ an order in $A$. Then
 $$\U (Z(\mathcal{O} )) = T\times F,$$
where $T$ is a finite group and $F$ is a free abelian group of rank $r_{\R}(A) -r_{\Q}(A)$. 

If $G$ is a finite group 
then for any finite field extension $K$ of $\Q$, $r_{K}(K G)$ is the number of irreducible $K$-characters of $G$ and it also equals the number of
Wedderburn components of $K G$.

Hence 
  $$Z(\U (\Z G)) = \pm Z(G) \times F,$$
where $F$ is a free abelian group of rank $r_{\R}(\R G)  - r_{\Q}(\Q G)$.

In particular, if $G$ is a finite abelian group of order $n$. Then $F$ has rank
 $$\frac{n+1+k_2 -2c}{2} = \sum_{d|n,\; d>2} k_d \left( \frac{\varphi (d)}{2} -1\right),$$
where $c$ is the number of cyclic subgroups of $G$ and $k_d$ is the number of cyclic subgroups of $G$ of order $d$.
\end{theorem}

A result of Artin says that 
if $G$ is a finite group then $r_{\Q}(\Q G)$, the number of irreducible $\Q$-characters of $G$, equals the number of conjugacy classes of cyclic subgroups of $G$.

As a consequence one obtains the following formula for the rank of the central units in a group ring.

\begin{corollary}
Let $G$ be a finite group. Then, the rank of $Z(\U (\Z G))$ is
 $$\frac{c+ c'}{2} -d,$$
where $c'$ is the number of conjugacy classes of $G$ closed under taking inverses and $d$ is the number of conjugacy classes of cyclic subgroups of $G$
\end{corollary}

Ritter and Sehgal determined necessary and sufficient conditions for all central units to be trivial.
A proof relies on the following lemma.

The following notation is used.
Let $G$ be a finite group and $K$ a field. One says that two elements $g$ and $h$ of $G$ are {\it $K$-conjugate} in $G$ if there exists
 $r\in \U_{K}(n)=\{ r \in \Z_n \mid \sigma (\xi_n ) =\xi_n^r, \; \text{ for some } \sigma \in \text{ Gal}(K(\xi_n )/K) \}$  ($\xi_n$ a primitive 
 $n$-th root of unity in an extension of $K$)
such that $g$ and $h^r$ are conjugate in $G$; where $n$ is the exponent of $G$. This defines an equivalence relation 
$\sim_K$ in $G$. The equivalence class containing $g\in G$ is called the {\it $K$-conjugacy class} of $g$ in $G$ and it is denoted $g_{K}^{G}$.
The conjugacy class of $g$ in $G$ is simply denoted $g^G$. Hence,
 $$g_{K}^{G} =\cup_{r\in \U_K(n)}(g^r)^{G}.$$
Note that if  $K$ contains a primitive $n$-th root of unity, then $g_K^{G}=g^G$. Further note that
$g\sim_{\Q} h$ if and only if  $g$ is conjugate of $h^r$ in $G$ for some $r$ coprime with $n$; equivalently $\langle g \rangle$ is a conjugate of $\langle h \rangle$ in $G$. One can also easily verify that $g\sim_{\R} h$ if and only if $g$ is a conjugate of $h$ or $h^-1$, that is $g_{\R}^{G} =g^{G}\cup (g^{-1})^{G}$.

\begin{lemma}
Let $G$ be a finite group of exponent $n$ and let $g\in G$. Then $g_{\Q}^{G} =g_{\R}^{G}$ if and only if $g$ is conjugate to $g^m$ or $g^{-m}$ for every integer $m$ with $(m,n)=1$.
\end{lemma}

\begin{corollary} (Ritter and Sehgal)
For a finite group $G$ the following properties are equivalent.
\begin{enumerate}
\item $Z(\U (\Z G))$ is finite (or equivalently, $G$ is a cut group), i.e. all central units are trivial.
\item For every $g\in G$ and every integer $m$ with $(m,|G|)=1$ the elements $g^m$ and $g^{-m}$ are conjugate.
\end{enumerate}
\end{corollary}

Representation theoretically cut groups are those groups  such  that the character fields are either the rationals or a quadratic imaginary extensions over $\Q$.
So, for example, rational groups are cut. Recently, cut groups gained in interest, but especially the subclass of rational groups has already a long tradition in classical representation theory. We refer the reader to for example
\cite{Bachle,BMPcentral,CD}.

Also for strongly monomial groups one can determine a formula for the rank of the central units
and, with some restriction, one can determine an independent set of central units that generates a subgroup of finite index.

We have seen that for many finite groups the group generated by the Bass units and the bicyclic units generate a subgroup of finite index
in $\U (\Z G)$. In particular, the subgroup contains a subgroup of finite index in the center $Z (\U (\Z G))$. As the bicyclic units contain a subgroup that only ``contributes'' to a subgroup  of finite
index in reduced norm one subgroups of orders in the simple components, one might be tempted to think that the Bass units contain a subgroup of finite
index in the center of $\U (\Z G)$. Note, however, that Bass units in general are not central elements.

Jespers, Parmenter and Sehgal showed that for finite nilpotent groups the group generated by the Bass units contains a subgroup of finite index in the unit group of the center. To do so, one needs, in first instance,  a method to construct from a Bass unit a central unit.
Jespers, Olteanu, Van Gelder and del R\'io proved that this also can be done for the class of abelian-by-supersolvable groups $G$ such that every cyclic subgroup of order not a divisor of $4$ or $6$ is subnormal in $G$.
Obviously,  dihedral groups are  examples of such groups. Also  nilpotent finite  groups $N$ are examples.
Indeed, let $Z_i=Z_i(N)$ denote the $i$-th center of $N$, i.e $Z_0=\{ 1\}$ and $Z_i/Z_{i-1}=Z(G/Z_{i-1})$ for $i\geq 1$. Then, for $x\in N$, the series
$\langle x \rangle  \lhd \langle Z_1, x \rangle  \lhd \cdots \lhd \langle Z_n, x \rangle  =N$ (for some integer $n$) is a subnormal series in $N$.

So, suppose $G$ is an finite abelian-by-supersolvable group such that every cyclic subgroup of order not a divisor of $4$ or $6$ is subnormal in $G$.
Let $g\in G$ be of order not a divisor or $4$ or $6$ and let
$$\mathcal{N} : N_0=\langle g \rangle \lhd N_1 \lhd N_2 \lhd \cdots \lhd N_m =G$$
be a subnormal series in $G$. For $u\in \U (\Z \langle g \rangle )$ put
 $$c_o^{\mathcal{N}}(u)=u$$
 and
  $$c_i^{\mathcal{N}} (u)=\prod_{h\in T_i}c_{i-1}^{\mathcal{N}}(u)^{h},$$
where $T_i$ is a transversal for $N_{i-1}$ in $N_i$, $i\geq 1$.
That this construction is well defined follows from the following lemma.

\begin{lemma}\label{centralconstruction}
With notation as above.
\begin{enumerate}
\item$ c_{i-1}^{\mathcal{N}} (u)^x \in \Z N_{i-1}$ for $x\in N_i$,
\item $c_{i-1}^{\mathcal{N}} (u)^{x}= c_{i-1}^{\mathcal{N}} (u)$ for $x\in N_{i-1}$,
\item $c_i^{\mathcal{N}} (u)$ is independent of the chosen transversal $T_i$.
\end{enumerate}
In particular, $c_m^{\mathcal{N}} (u)\in Z(\U (\Z G))$.
\end{lemma}

Because the class of abelian-by-supersolvable groups is closed under taking subgroups (a property that does not hold for the larger class consisting of
strongly monomial groups) one can prove the following result.

\begin{theorem} (Jespers, Olteanu, del R\'{\i}o and Van Gelder) 
Let $G$ be a finite abelian-by-supersolvable group such that every cyclic subgroup of order not a divisor of $4$ or $6$ is subnormal in $G$.
Let $g\in G$ be of order not a divisor or $4$ or $6$. Then, the group generated by the Bass units of $\Z G$ contains a subgroup of finite index in $Z (\Z (\U (\Z G)))$.

Actually, for each subgroup $\langle g \rangle$, of order not dividing $4$ or $6$, fix a subnormal series $\mathcal{N}_{g}$ from $\langle g \rangle$ to $G$. Then
 $$\langle c^{\mathcal{N}_{g}} (b_g) \mid b_g \text{ a Bass unit based on } g, \; g\in G\rangle ,$$
 is of finite index in $Z(\U (\Z G))$.
\end{theorem}

Recently, a beautiful generalization of this result has been obtained by Bakshi and Kaur \cite{BakshiKaur2} for a uch wider class of groups, the class consisting of the  generalized strongly monomial groups (which defined via the notion  of generalized strong Shoda pair).

\section{Structure theorems of unit groups} 

The exceptional simple components are an obstruction for the construction of finitely many generators for a subgroup of finite index in the unit group of $\Z G$ (for a finite group $G$).
Maybe surprisingly, many of these components are not an obstruction for proving a ``structure theorem'', on the contrary.

According to Kleinert\cite{kleinert}  a ``Unit Theorem''  for the unit group $\U (\Z G)$ is a statement that should at least  consist, in purely group theoretical terms, of a class of groups $\mathcal{G}$ such that almost all torsionfree subgroups
of finite index  in $\U (\Z G)$ are members of $\mathcal{G}$.

So one can pose the following general problem.

\begin{problem}
For a class of groups $\mathcal{G}$, classify the finite groups $G$, such that $\U (\Z G)$ constains a subgroup of finite index in  $\mathcal{G}$.
\end{problem}

In the following results we state the answer for the class of groups  $\mathcal{G}$ that consists of 
direct products of free products of abelian groups  (Jespers and del R\'io) and for the class of groups that consists of the direct products of free-by-free groups (Jespers, Pita, del R\'io, Ruiz, P. Zalesskii).

\begin{theorem}
The following properties are equivalent for a finite group $G$.
\begin{enumerate}
\item $\U (\Z G)$ is either virtually abelian or virtually nonabelian free.
\item $\U (\Z G)$ is virtually a free product of abelian groups.
\item $\Q G$ is a direct product of fields, division rings of the form $\left( \frac{-1,-3}{\Q} \right)$, or $\mathbb{H}(K)$ with $K=\Q$, $\Q (\sqrt{2})$ or $\Q(\sqrt{3})$ and at most one copy of $M_{2}(\Q )$.
\item One of the following conditions hold:
\begin{enumerate}
\item $G=Q_8 \times C_2^{n}$,
\item $G$ is abelian,
\item $G$ is one of the following groups: $D_6,\; D_8, Q_{12} =\langle a,b\mid a^6=1,\; b^2=a^3,\; ba=a^5b\rangle$, $P=\langle a,b\mid a^4=1, \; b^4=1,
aba^{-1} b^{-1}= a^2\rangle$ (in this case $\U (\Z G)$ is virtually nonabelian free).
\end{enumerate}
\end{enumerate}
\end{theorem}

Note that the respective Wedderburn decomposition of the mentioned   rational group algebras is as follows.
\begin{eqnarray*}
{\Q} D_{6} & \cong & 2{\Q}\oplus M_{2}({\Q}),\\
{\Q} D_{8} & \cong & 4{\Q}\oplus M_{2}({\Q}),\\
{\Q}Q_{8} & \cong & 4{\Q} \oplus {\mathbb{H}}({\Q}),\\
{\Q} P & \cong & 4{\Q} \oplus 2{\Q}(i)
           \oplus {\mathbb{H}}({\Q}) \oplus M_{2}({\Q})\\
{\Q} Q_{12} & \cong & 2\Q \oplus \Q (\sqrt{-3}) +\left( \frac{-1,-3}{\Q}\right)  \oplus M_{2}({\Q}),\\
\end{eqnarray*}

\begin{theorem}
The following properties are equivalent for a finite group $G$.
\begin{enumerate}
\item $\U (\Z G)$ is virtually a direct product of free-by-free groups.
\item For every simple component $A$ of $\Q G$ and some (every) order $\mathcal{O}$ in $A$, the group of reduced norm one elements in $\mathcal{O}$
is virtually free-by-free.
\item Every simple component of $\Q G$ is either a field, a totally definite quaternion algebra, or $M_{2}(K)$ where  $K$ is either $\Q (i)$, $\Q (\sqrt{-2})$,
$\Q (\sqrt{-3})$.
\item $G$ is either abelian or an epimorphic image of $A\times H$, where $A$ is an abelian group and one of the following conditions holds:
\begin{enumerate}
\item $A$ has exponent $6$ and $H$ is one of the groups $\mathcal{W}$, $\mathcal{W}_{1n}$ or $\mathcal{W}_{2n}$.
\item $A$ has exponent $4$ and $H$ is one of the groups $\mathcal{V}$, $\mathcal{V}_{1n}$, $\mathcal{V}_{2n}$, $\mathcal{U}_1$ or $\mathcal{U}_2$.
\item $A$ has exponent $2$ and $H$ is one of the group $\mathcal{T}$, $\mathcal{T}_{1n}$, $\mathcal{T}_{2n}$ or $\mathcal{T}_{3n}$.
\item $H=M\rtimes P = (M\times Q):\langle \overline{u}\rangle_{2}$, where $M$ is an elementary abelian $3$-group, $P=Q:\langle \overline{u}\rangle_2$, $m^u=m^{-1}$ for every $m\in M$, and one of the following conditions holds:
\begin{enumerate}
  \item[-] $A$ has exponent $4$ and $P=C_8$,
  \item[-]  $A$ has exponent $6$, $P=W_{1n}$ and $Q=\langle y_1, \ldots , y_n , t_1, \ldots , t_n ,x^2\rangle$,
  \item[-]  $A$ has exponent $2$, $P=W_{21}$ and $Q=\langle y_1^2 , x\rangle$.
\end{enumerate}
\end{enumerate}
\end{enumerate}
The non-nilpotent groups are those listed in (4) with $M$ non-trivial.
\end{theorem}

The first class consists of the following groups.
\begin{eqnarray*}
 \mathcal{W} &=& \left(\langle t \rangle_2 \times \langle x^2 \rangle_2 \times \langle y^2\rangle_2\right) : \left(\langle \overline{x}\rangle_2 \times
\langle \overline{y}\rangle_2\right),\\
     && \text{ with } t = \left(x,y\right)  \text{ and } Z(\mathcal{W})=\langle x^2,y^2,t\rangle.\\
 \mathcal{W}_{1n} &=& \left(\prod\limits_{i=1}^n \langle t_i\rangle_2 \times \prod\limits_{i=1}^n \langle y_i\rangle_2 \right) \rtimes \langle x\rangle_4,\\
 &&\text{ with } t_{i} = (x,y_i) \text{ and } Z(\mathcal{W}_{1n})=\langle t_1,\dots,t_n,x^2\rangle.\\
 \mathcal{W}_{2n} &=& \left(\prod\limits_{i=1}^n \langle y_i\rangle_4 \right) \rtimes \langle x\rangle_4,\\
   &&\text{ with } t_{i} = (x,y_i) = y_i^2 \text{ and } Z(\mathcal{W}_{2n})=\langle t_1,\dots,t_n,x^2\rangle.
\end{eqnarray*}

The second class of groups consists of the following groups.

\begin{eqnarray*}
\mathcal{V} &=& \left(\langle t\rangle_2 \times \langle x^2\rangle_4 \times \langle y^2\rangle_4\right) :
 \left(\langle \overline{x}\rangle_2 \times \langle \overline{y}\rangle_2\right),\\
 &&\text{ with } t = (x,y) \text{ and } Z(\mathcal{V})=\langle x^2,y^2,t\rangle .\\
 \mathcal{V}_{1n} &=& \left(\prod\limits_{i=1}^n \langle t_i\rangle_2 \times \prod\limits_{i=1}^n \langle y_i\rangle_4
\right) \rtimes \langle x\rangle_8,\\
 && \text{ with } t_{i} = (x,y_i) \text{ and }Z(\mathcal{V}_{1n})=\langle t_1,\dots,t_n,y_1^2,\dots,y_n^2,x^2\rangle.\\
 \mathcal{V}_{2n} &=& \left(\prod\limits_{i=1}^n \langle y_i\rangle_8 \right) \rtimes \langle x\rangle_8,\\
 &&\text{ with } t_{i} = (x,y_i) = y_i^4 \text{ and } Z(\mathcal{V}_{2n})=\langle t_1,\dots,t_n,x^2\rangle.
 \end{eqnarray*}

The third class consists of the following groups.

\begin{eqnarray*}
  \mathcal{U}_1 &=& \left(\prod\limits_{1\le i< j \le 3} \langle t_{ij}\rangle_2 \right) :\left(\prod\limits_{k=1}^3 \langle \overline{y_k}\rangle_4\right), \\
&& \text{with } Z(\mathcal{U}_1)=\langle t_{12},t_{13},t_{23},y_1^2,y_2^2,y_3^2\rangle, \;
 t_{ij} = (y_i,y_j) \text{ and } y_i^4=1 .\\
\mathcal{U}_2 &=& \left(\prod\limits_{1\le i< j \le 3} \langle t_{ij}\rangle_2  \right):\left(\prod\limits_{k=1}^3 \langle \overline{y_k}\rangle_4\right),\\
&& \text{with } Z(\mathcal{U}_2)=\langle t_{12},t_{13},t_{23},y_1^2,y_2^2,y_3^2\rangle, \; t_{ij} = (y_i,y_j),\\
&& y_1^4=1,\;y_2^4=t_{12}\text{ and } y_3^4=t_{13}.
\end{eqnarray*}

The following groups form part of the fourth class of groups.

\begin{eqnarray*}
 \mathcal{T}_{1n} &=&  \left(\prod\limits_{i=1}^n \langle t_i\rangle_4 \times \prod\limits_{i=1}^n \langle y_i\rangle_4 \right) \rtimes
                 \langle x\rangle_8,\\
 && \text{ with } t_{i} = (x,y_i),\; (x,t_i)=t_i^2 \text{ and } Z(\mathcal{T}_{1n})=\langle t_1^2,\ldots,t_n^2,x^2\rangle.
\end{eqnarray*}

A major issue remains the lack of knowledge of constructung large subgroups of the unit group of an order in a finite dimensional rational division algebra (so dealing with orders in exceptional components of type 1).

\begin{problem}
Discover generic constructions of units in orders of division algebras that are simple components of a rational group algebra $\Q G$ of a finite group. Discover generators of large subgroups in such orders.
\end{problem}

\begin{problem}
Describe finitely many generators for the following unit groups:
 $$\U (\Z (Q_8 \times C_3)) \quad \text{ and } \quad \U (\Z (Q_8 \times C_7 )).$$
\end{problem}

In recent work by B\"achle, Janssens, Jespers, Kiefer, Temmerman \cite{BJJKT}, it has been investigated when the unit group $\U (\Z G)$ (or more specifically the group generated by the bicyclic units) can or cannot be decomposed into a non-trivial amalgamated product. 
This is done  under the assumption that $\U (\Z G)^{ab}$ is finite. 
Because $\U (\Z G)$ is a finitely generated group, a result of Serre \cite{Serre} says that 
being both not a non-trivial amalgamated product and $\U (\Z G)^{ab}$ finite precisely occurs  when $\U (\Z G)$ has property FA. Recall that a group has property FA if every action on 
a simplicial tree has a global fixed point. Since unit theorems concern a property on all subgroups of finite index, 
one considers the hereditary property, denoted HFA, and a finite abelianization property, denoted FAb.  One 
says that a group has 
HFA if all its finite index subgroups have property FA and one says that a group  has property FAb if 
every subgroup of finite index has finite abelianization.
It is well-known that property FA follows from Kazhdan's property T (see \cite{BHV}).
Recall from Delorme-Guichardet's Theorem \cite[Theorem 2.12.4]{BHV} that a countable discrete group $\Gamma$  has property (T) if and only if every affine isometric action of $\Gamma$  on a real Hilbert space has a fixed point. 

B\"achle, Janssens, Jespers, Kiefer, Temmerman proved in \cite{BJJKT} 
 a   characterization  of when $\U (\Z G)$ satisfies  these hereditary properties.
Surprisingly, all these fixed point properties are equivalent and are controlled both in terms of $G$ and in terms of the Wedderburn decomposition of $\Q G$.

\begin{theorem}  
Let $G$ be a finite group. The following properties are equivalent:
\begin{enumerate}
\item The group $\U(\Z G)$ has property HFA,
\item The group $\U(\Z G)$ has property T,
\item The group $\U(\Z G)$ has property FAb,
\item $G$ is cut and $\mathbb{Q}G$ has no exceptional components,
\item $G$ is cut and $G$ does not map onto one of $10$ explicitly described groups.
\end{enumerate}
In particular, if these conditions are satisfied, then the group generated by the bicyclic units is not a non-trivial amalgamated product.

Furthermore, if $G$ does not have exceptional simple components (e.g. $G$ is of odd order), then 
the above conditions are equivalent to the following two equivalent conditions
\begin{enumerate}
\item[6.]
 $\mathcal{U}(\mathbb{Z}G)^{ab}$ is finite, 
 \item[7.]  $G$ is a cut group.
 \end{enumerate}
\end{theorem}

In \cite{BJJKT1}, B\"achle, Janssens, Jespers, Kiefer, Temmerman proved  the following dichotomy.

\begin{theorem} 
Let $G$ be finite group which is solvable or $5 \not| |G|$. If $G$ is a cut group, then exactly one of the following properties holds:
\begin{enumerate}
\item $\U(\Z G)$ has property T.
\item $\U(\Z G)$ is commensurable with a non-trivial amalgamated product.
\end{enumerate}
\end{theorem}

In \cite{BJJKT1}, B\"achle, Janssens, Jespers, Kiefer, Temmerman also  proved  the following unit theorem.

\begin{theorem}
Let $G$ be a finite group having $D_8$ or $S_3$ as an epimorphic image. Then $\U(\Z G)$ is virtually a non-trivial amalgamated product. 
\end{theorem}

\end{document}